\documentclass[a4paper,reqno,final]{amsart}
\newif\ifdissertationkrenn\dissertationkrennfalse

\usepackage[eso-foot]{svninfo}
\svnInfo $Id: wnaf-analysis-hyperell.tex 413 2013-03-13 18:31:09Z krenn@MATH.TU-GRAZ.AC.AT $ 

\usepackage[utf8]{inputenc}
\usepackage{ae,aecompl}   
\usepackage[english]{babel}

\usepackage[wnafs,cellrounding,setmiddle=colon]{mathdef}

\makeatletter \renewcommand\p@enumii{} \makeatother
\newcommand{\itemref}[1]{(\ref{#1})}

\newcommand{\fracpartLambda}[1]{\ensuremath{\fracpart{#1}_{\Lambda}}}
\newcommand{\diag}[1]{\ensuremath{\f{\operatorname{diag}}{#1}}}

\newcommand{\colbw}{bw}

\begin{document}


\title[Analysis of the Width-$w$ Non-adjacent Form]{Analysis of the Width-$w$
  Non-Adjacent Form in Conjunction with Hyperelliptic Curve Cryptography and
  with Lattices}

\author{Daniel Krenn}

\thanks{The author is supported by the Austrian Science Fund (FWF): S9606,
  that is part of the Austrian National Research Network ``Analytic
  Combinatorics and Probabilistic Number Theory'', and by the Austrian Science
  Fund (FWF): W1230, Doctoral Program
  ``Discrete Mathematics''.}

\address{\parbox{12cm}{%
    Daniel Krenn \\
    Institute of Optimisation and Discrete Mathematics (Math B) \\
    Graz University of Technology \\
    Steyrergasse 30/II, A-8010 Graz, Austria \\}} 

\email{\href{mailto:math@danielkrenn.at}{math@danielkrenn.at} \textit{or}
  \href{mailto:krenn@math.tugraz.at}{krenn@math.tugraz.at}}




\begin{abstract}


  In this work the number of occurrences of a fixed non-zero digit in the
  width\nbd-$w$ non-adjacent forms of all elements of a lattice in some region
  (e.g.\ a ball) is analysed. As bases, expanding endomorphisms with
  eigenvalues of the same absolute value are allowed. Applications of the main
  result are on numeral systems with an algebraic integer as base. Those come
  from efficient scalar multiplication methods (Frobenius-and-add methods) in
  hyperelliptic curves cryptography, and the result is needed for analysing the
  running time of such algorithms.

  The counting result itself is an asymptotic formula, where its main term
  coincides with the full block length analysis. In its second order term a
  periodic fluctuation is exhibited. The proof follows Delange's method.

\end{abstract}


\maketitle


\section{Introduction}
\label{sec:introduction}


One main operation in hyperelliptic curve cryptography is building (large)
multiples of an element of the Jacobian variety of a hyperelliptic curve over a
finite field. Clearly, we want to perform that scalar multiplication as
efficiently as possible. A standard method there are double-and-add algorithms,
where integers are written in binary, and then a Horner scheme is performed. By
using windowing methods those algorithms can be sped up. The idea is to take a
larger digit set and choose an expansion which has a low number of non-zero
digits. This leads to an efficient evaluation. Some background information on
hyperelliptic curve cryptography can be found for example
in~\cite{Avanzi-Cohen-Doche-Frey:2005:handb-ellip}.

If the hyperelliptic curve is defined over a finite field with $q$ elements and
we are working over an extension (over a field with $q^m$ elements), then one
can use a Frobenius-and-add method instead. There the (expensive) doublings are
replaced by the (cheap) evaluation of the Frobenius endomorphism on the
Jacobian variety: If
\begin{equation*}
  z = \sum_{\ell=0}^{L-1} \xi_\ell \tau^\ell
\end{equation*}
with digits $\xi_\ell$ and where the base~$\tau$ is a zero of the
characteristic polynomial of the Frobenius endomorphism on the Jacobian, then
for an element $Q$ of the Jacobian we can compute $zQ$ by
\begin{equation*}
  zQ = \sum_{\ell=0}^{L-1} \xi_\ell \f{\varphi^\ell}{Q},
\end{equation*}
where $\varphi$ denotes the Frobenius endomorphism. That base~$\tau$ is an
algebraic integer whose conjugates all have the same absolute value, cf.\
Deligne~\cite{Deligne:1974}, Dwork~\cite{Dwork:1960} and
Weil~\cite{Weil:1948:var-ab-et-courbes-alg, Weil:1949,
  Weil:1971:courbes-alg-et-var-ab}, and see Section~\ref{sec:thm-hyperell} for
more details.

So let us consider digit expansions with a base as above. Let $w$ be a
positive integer. Our digit set should consist of $0$ and one representative of
every residue class modulo $\tau^w$ which is not divisible by $\tau$. That
choice of the digit set yields redundancy, i.e., each element of
$\Ztau$ has more than one representation. The width\nbd-$w$ non-adjacent
form, $w$-NAF for short, is a special representation: Every block of $w$
consecutive digits contains at most one non-zero digit. The choice of the digit
set guarantees that the $w$-NAF-expansion is unique. The low weight (number of
non-zero digits) of that expansion makes the arithmetic on the hyperelliptic
curves efficient.

In the case that the base~$\tau$ is an imaginary-quadratic algebraic integer,
properties of such \wNAF{} numeral systems are known: The question whether for
a given digit set each element of $\Ztau$ has a representation as a \wNAF{} is
investigated in Koblitz~\cite{Koblitz:1998:ellip-curve},
Solinas~\cite{Solinas:1997:improved-algorithm,Solinas:2000:effic-koblit},
Blake, Murty and Xu~\cite{Blake-Kumar-Xu:2005:effic-algor,
  Blake-Murty-Xu:ta:nonad-radix, Blake-Murty-Xu:2005:naf}, and Heuberger and
Krenn~\cite{Heuberger-Krenn:2012:wnaf-analysis}. Another question, namely
whether the \wNAF{} is an expansion which minimises the weight among all
possible expansions with the same digit set, is answered in Heuberger and
Krenn~\cite{Heuberger-Krenn:2011:wnafs-optimality}. A generalisation of those
existence and optimality results to higher degree of the base~$\tau$ is given
in Heuberger and Krenn~\cite{Heuberger-Krenn:2013:general-wnaf}. One main step
there was to use the Minkowski map to transform the $\tau$\nbd-adic setting to
a lattice, see also Section~\ref{sec:thm-hyperell}.

The present work deals with analysing the number of occurrences of a digit in
\wNAF{}-expansions with base~$\tau$ (an algebraic integer of degree~$n$) and
where $w$ is chosen sufficiently large. This result is needed for the analysis
of the running time of the scalar multiplication algorithm mentioned at the
beginning of this introduction. As brought up in the previous paragraph, we
will do this analysis in the set-up of numeral systems in lattices, cf.\
Section~\ref{sec:thm-hyperell}. As a base, an expanding endomorphism, whose
eigenvalues all have the same absolute value, is used. Our main result is the
asymptotic formula
\begin{equation*}
  Z_\eta = N^n \tfrac{\pi^{n/2}}{\f{\Gamma}{\frac{n}{2}+1}} E \log N 
  + N^n \f{\psi_\eta}{\log N} + \Oh{N^\beta \log N}.
\end{equation*}
for the number of occurrences of a fixed non-zero digit~$\eta$ in
\wNAF{}-expansions in a ball around $0$ with radius~$N$. The main term of that
formula coincides with the full block length analysis given in Heuberger and
Krenn~\cite{Heuberger-Krenn:2012:wnaf-analysis}. There an explicit expression
for the expectation (the constant~$E$) and the variance of the occurrence of
such a digit in all expansions of a fixed length is given. The result here is
more precise: A periodic fluctuation $\psi_\eta$ in the second order term is
also exhibited. The third term is an error term with $\beta<n$. Such
structures---main term, oscillation term, smaller error term---are not uncommon
in the context of digits counting, see for instance, Heuberger and
Prodinger~\cite{Heuberger-Prodinger:2006:analy-alter} or Grabner, Heuberger and
Prodinger~\cite{Grabner-Heuberger-Prodinger:2004:distr-results-pairs}. The
result itself is a generalisation of the one found in Heuberger and
Krenn~\cite{Heuberger-Krenn:2012:wnaf-analysis}. The proof, as the one
in~\cite{Heuberger-Krenn:2012:wnaf-analysis}, follows Delange's method, cf.\
Delange~\cite{Delange:1975:chiffres}, but several technical problems have to be
taken into account.

The structure of this article is as follows. We start with the formal
definition of numeral systems and the non-adjacent form in
Section~\ref{sec:nafs}. Sections~\ref{sec:set-up}
and~\ref{sec:remarks} contain our primary set-up in a lattice. We will
work in this set-up throughout the entire article. There also the used
digit set, which comes from a tiling by the lattice, is
defined. Additionally, some notations are fixed and some basic
properties are given. The end of Section~\ref{sec:set-up} is devoted
to the full block length analysis theorem given in Heuberger and
Krenn~\cite{Heuberger-Krenn:2012:wnaf-analysis}. In
Sections~\ref{sec:bounds-value} to~\ref{sec:sets-w_eta} a lot of
properties of the investigated expansions, such as bounds of the value
and the behaviour of the fundamental domain and the characteristic
sets, are derived. Those are needed to prove our main result, the
counting theorem in Section~\ref{sec:counting-digits-region}. The last
section will forge a bridge to the $\tau$-adic set-up. This is
explained with details there and the counting theorem is restated in
that set-up.

A last remark on the proofs given in this article. As this work is a
generalisation of Heuberger and
Krenn~\cite{Heuberger-Krenn:2012:wnaf-analysis} several proofs of
propositions and lemmata are skipped. All those are straight-forward
generalisations of the ones for the quadratic case, which means, we
have to do things like replacing $\Ztau$ by the lattice, the
multiplication by $\tau$ by a lattice endomorphism, the dimension~$2$
by~$n$, using a norm instead of the absolute value, and so on. If the
generalisation is not that obvious, the proofs are given.


\section{Non-Adjacent Forms}
\label{sec:nafs}


This section is devoted to the formal introduction of width\nbd-$w$ non-adjacent
forms. Let $\Lambda$ be an Abelian group, $\Phi$ an injective endomorphism of
$\Lambda$ and $w$ a positive integer. Later, starting with the next section,
the group $\Lambda$ will be a lattice with the usual addition of lattice
points.

We start with the definition of the digit set used throughout this article.


\begin{definition}[Reduced Residue Digit Set]\label{def:red-residue-digit-set}
  Let $\cD\subseteq\Lambda$. The set $\cD$ is called a \emph{reduced residue
    digit set modulo $\Phi^w$}, if it is consists of $0$ and exactly one
  representative for each residue class of $\Lambda$ modulo $\Phi^w\Lambda$
  that is not contained in $\Phi\Lambda$.
\end{definition}


Next we define the syntactic condition of our expansions. This syntax is used
to get unique expansions, because our numeral systems are redundant.


\begin{definition}[Width-$w$ Non-Adjacent Forms]\label{def:wnaf}
  Let $\bfeta = \sequence{\eta_j}_{j\in\Z} \in \cD^\Z$. The sequence $\bfeta$ is
  called a \emph{width\nbd-$w$ non-adjacent form}, or
  \emph{\wNAF{}} for short, if each factor $\eta_{j+w-1}\ldots\eta_j$, i.e.,
  each block of width~$w$, contains at most one non-zero digit.

  Let $J\colonequals\set*{j\in\Z}{\eta_j\neq0}$. We call $\f{\sup}{\set{0} \cup
    (J+1)}$, where $J+1=\set*{j+1}{j\in J}$, the \emph{left-length of the
    \wNAF{}~$\bfeta$} and $-\f{\inf}{\set{0} \cup J}$ the \emph{right-length of
    the \wNAF{} $\bfeta$}. Let $\ell$ and $r$ be elements of $\N_0 \cup
  \set{\wNAFsetFIN,\infty}$, where $\wNAFsetFIN$ means finite. We denote the
  \emph{set of all \wNAF{}s of left-length at most $\ell$ and right-length at
    most $r$} by $\wNAFsetellell{\ell}{r}$. The elements of the set
  $\wNAFsetfinell{0}$ will be called \emph{integer \wNAF{}s}. The
  \emph{most-significant digit} of a $\bfeta\in\wNAFsetfininf$ is the digit
  $\eta_j\neq0$, where $j$ is chosen maximally with that property.

  For $\bfeta\in\wNAFsetfininf$ we call
  \begin{equation*}
    \NAFvalue{\bfeta} \colonequals \sum_{j\in\Z} \Phi^j \eta_j 
  \end{equation*}
  the \emph{value of the \wNAF{}~$\bfeta$}.
\end{definition}


The following notations and conventions are used. A block of any number of zero
digits is denoted by $\bfzero$. For a digit $\eta$ and $k\in\N_0$ we will use
\begin{equation*}
  \eta^k\colonequals \underbrace{\eta\ldots\eta}_{k},
\end{equation*}
with the convention $\eta^0 \colonequals \eps$, where $\eps$ denotes the empty
word.  A \wNAF{} $\bfeta = \sequence{\eta_j}_{j\in\Z}$ will be written as
$\bfeta_I\bfldot\bfeta_F$, where $\bfeta_I$ contains the $\eta_j$ with $j\geq0$
and $\bfeta_F$ contains the $\eta_j$ with $j<0$. $\bfeta_I$ is called
\emph{integer part}, $\bfeta_F$ \emph{fractional part}, and the dot is called
\emph{$\Phi$\nbd-point}. Left-leading zeros in $\bfeta_I$ can be skipped,
except $\eta_0$, and right-trailing zeros in $\bfeta_F$ can be skipped as
well. If $\bfeta_F$ is a sequence containing only zeros, the $\Phi$\nbd-point
and this sequence are not drawn.

Further, for a \wNAF{} $\bfeta$ (a bold, usually small Greek letter) we will
always use $\eta_j$ (the same letter, but indexed and not bold) for the
elements of the sequence.


The set $\wNAFsetfininf$ can be equipped with a metric. It is defined in the
following way. Let $\rho>1$. For $\bfeta\in\wNAFsetfininf$ and
$\bfxi\in\wNAFsetfininf$ define
\begin{equation*}
  \NAFd{\bfeta}{\bfxi} \colonequals 
  \begin{cases}
    \rho^{\max\set*{j\in\Z}{\eta_j\neq\xi_j}} 
    & \text{if $\bfeta\neq\bfxi$,} \\
    0 & \text{if $\bfeta=\bfxi$.}
  \end{cases}
\end{equation*}
So the largest index, where the two \wNAF{}s differ, decides their distance. See
for example Edgar~\cite{Edgar:2008:measur} for details on such metrics.

We get a compactness result on the metric space $\wNAFsetellinf{\ell} \subseteq
\wNAFsetfininf$, $\ell\in\N_0$, see the proposition below. The metric space
$\wNAFsetfininf$ is not compact, because if we fix a non-zero digit $\eta$,
then the sequence $\sequence{\eta0^j}_{j\in\N_0}$ has no convergent
subsequence, but all $\eta0^j$ are in the set $\wNAFsetfininf$.


\begin{proposition}\label{pro:naf-compact}
  For every $\ell\geq0$ the metric space
  $\left(\wNAFsetellinf{\ell},\NAFdname\right)$ is compact.
\end{proposition}


This is a consequence of Tychonoff's Theorem, see~\cite{Heuberger-Krenn:2012:wnaf-analysis} for details.


\section{The Set-Up and Notations}
\label{sec:set-up}


In this section we describe the set-up, which we use throughout this
article. 


\begin{enumerate}
\item Let $\Lambda$ be a lattice in $\R^n$ with full rank, i.e., $\Lambda =
  w_1\Z \oplus \dots \oplus w_n\Z$ for linearly independent
  $w_1,\dots,w_n\in\R^n$.

\item Let $n\in\N$ and $\Phi$ be an endomorphism of $\R^n$ with
  $\f{\Phi}{\Lambda} \subseteq \Lambda$. We assume that each eigenvalue of
  $\Phi$ has the same absolute value $\rho$, where $\rho$ is a fixed real
  constant with $\rho>1$. Further we assume that $\rho^n\in\N$. Additionally,
  we take this $\rho$ as parameter in the definition of the metric $\NAFdname$.

\item Suppose that the set $T\subseteq\R^n$ tiles the space~$\R^n$ by the
  lattice~$\Lambda$, i.e., the following two properties hold:
  \begin{enumerate}
  \item $\bigcup_{z\in \Lambda}(z+T)=\R^n$,
  \item $T\cap (z+T)\subseteq \boundary*{T}$ holds for all $z\in \Lambda$ with
    $z\neq 0$.
  \end{enumerate}
  Further, we assume that $T$ is closed and that $\lmeas{\boundary*{T}}=0$,
  where $\lambda$ denotes the $n$\nbd-dimensional Lebesgue measure. We set
  $d_\Lambda \colonequals \lmeas{T}$.

\item Let $\norm{\,\cdot\,}$ be a vector norm on $\R^n$ such that for the
  corresponding induced operator norm, also denoted by $\norm{\,\cdot\,}$, the
  equalities $\norm{\Phi}=\rho$ and $\norm{\Phi^{-1}}=\rho^{-1}$ hold.

  For a $z\in\Lambda$ and non-negative $r\in\R$ the \emph{open ball with centre
    $z$ and radius $r$} is denoted by
  \begin{equation*}
    \ball{z}{r} \colonequals \set*{y\in\Lambda}{\norm{z-y} < r}
  \end{equation*}
  and the \emph{closed ball with centre $z$ and radius $r$} by
  \begin{equation*}
    \ball*{z}{r} \colonequals \set*{y\in\Lambda}{\norm{z-y} \leq r}.
  \end{equation*}

\item Let $r$ and $R$ be positive reals with
  \begin{equation}\label{eq:T-bounds}
    \ballc{0}{r} \subseteq T \subseteq \ballc{0}{R}.
  \end{equation}

\item Let $w$ be a positive integer such that
  \begin{equation}\label{eq:existence-condition}
    \frac{R}{r} < \rho^w-1.
  \end{equation}

\item Let $\cD$ be a reduced residue digit set modulo $\Phi^w$, cf.\
  Definition~\ref{def:red-residue-digit-set}, corresponding to
  the tiling~$T$, i.e.\ the digit set $\cD$ fulfils $\cD\subseteq \Phi^w T$. 

  Further, suppose that the cardinality of the digit set $\cD$ is
  \begin{equation*}
    \rho^{n(w-1)}\left(\rho^n-1\right)+1.
  \end{equation*}
\end{enumerate}


We use the following notation concerning our tiling: for a lattice element
$z\in\Lambda$ we set $T_z \colonequals z+T$. Therefore $\bigcup_{z\in \Lambda}
T_z=\R^n$ and $T_y \cap T_z \subseteq \boundary*{T_z}$ for all distinct $y,z\in
\Lambda$.


Next we define a fractional part function in $\R^n$ with respect to the
lattice~$\Lambda$, which should be a generalisation of the usual fractional
part of elements in~$\R$ with respect to the rational integers~$\Z$. Our tiling
$T$ induces such a fractional part.


\begin{definition}[Fractional Part]\label{def:frac-voronoi}
  Let $\wt{T}$ be a tiling arising from $T$ in the following way: Restrict the
  set $\wt{T} \subseteq T$ such that it fulfils $\biguplus_{z\in
    \Lambda}(z+\wt{T})=\R^n$.

  For $z\in\R^n$ with $z=u+v$, where $u\in\Lambda$ and $v\in\wt{T}$ define the
  \emph{fractional part corresponding to the lattice $\Lambda$} by
  $\fracpartLambda{z} \colonequals v$.
\end{definition}

Note that this fractional part depends on the tiling~$T$ (or more precisely, on
the tiling $\wt{T}$). We omit this dependency, since we assume that our tiling
is fixed.


\section{Some Basic Properties and some Remarks}
\label{sec:remarks}


The previous section contained our set-up. Some basic implications of that
set-up are now given in this section. Further we give remarks on the tilings
and on the digit sets used, and there are also comments on the existence of
\wNAF{}-expansions in the lattice.

We start with three remarks on our mapping~$\Phi$.


\begin{remark}
  Since all eigenvalues of $\Phi$ have an absolute value larger than $1$, the
  function $\Phi$ is injective. Note that we already assumed injectivity of the
  endomorphism~$\Phi$ in the basic definitions given in Section~\ref{sec:nafs}.
\end{remark}


\begin{remark}
  We have assumed $\norm{\Phi}=\rho$ and $\norm{\Phi^{-1}}=\rho^{-1}$.
  Therefore, for all $J \in \Z$ the equality $\norm{\Phi^J} = \rho^J$ follows.
\end{remark}


\begin{remark}
  The endomorphism~$\Phi$ is diagonalisable. This follows from the assumptions
  that all eigenvalues have the same absolute value~$\rho$ and the existence of
  a norm with $\norm{\Phi}=\rho$ as described in the in the paragraph below.

  Let $\Phi = Q^{-1} J Q$ be the Jordan decomposition of~$\Phi$ and assume the
  endomorphism~$\Phi$ is not diagonalisable. Then there is a Jordan block of~$J$
  of size at least~$2$. Therefore, by building~$\Phi^m$ for positive
  integers~$m$, we get $m\rho^{m-1}u_m$ with $\abs{u_m}=1$ as a superdiagonal
  entry of~$J^m$. Now choose a normalised vector~$x$ such that $J^mQx$ extracts
  (is equal to) a multiple of the column with that entry. That column has only
  the two entries $m\rho^{m-1}u_mc$ and $\rho^mv_mc$ with $\abs{v_m}=1$ and a
  constant~$c$. Therefore the norm of~$\Phi^mx=Q^{-1}J^mQx$ is bounded from
  below by $m\rho^md$ for an appropriate constant~$d>0$. Choosing~$m$ large
  enough leads to a contradiction, since $\norm{\Phi^mx}\leq \rho^m$.
\end{remark}


One special tiling comes from the Voronoi diagram of the lattice. This is
stated in the remark below.


\begin{remark}
  Let
  \begin{equation*}
    V \colonequals \set*{z\in\R^n}{\forall y\in\Lambda : 
      \norm{z} \leq \norm{z-y}}.
  \end{equation*}
  We call $V$ the \emph{Voronoi cell for $0$} corresponding to the
  lattice~$\Lambda$. Let $u \in \Lambda$. We define the \emph{Voronoi cell for
    $u$} as $V_u \colonequals u + V$.

  Now choosing $T=V$ results in a tiling of the~$\R^n$ by the
  lattice~$\Lambda$.
\end{remark}


In our set-up the digit set corresponds to the tiling. In
Remark~\ref{rem:digitset:tiling} this is explained in more details. The Voronoi
tiling mentioned above gives rise to a special digit set, namely the minimal
norm digit set. There, for each digit a representative of minimal norm is
chosen.


\begin{remark}\label{rem:digitset:tiling}
  The condition $\frac{R}{r} < \rho^w-1$ in the our set-up implies the
  existence of \wNAF{}s: each element of $\Lambda$ has a unique
  \wNAF{}-expansion with the digit set $\cD$. See Heuberger and
  Krenn~\cite{Heuberger-Krenn:2013:general-wnaf} for details. There, numeral
  systems in lattices with \wNAF{}-condition and digit sets coming from tilings
  are explained in detail. Further it is shown that each tiling and positive
  integer~$w$ give rise to a digit set~$\cD$.

  Because $\cD\subseteq \Phi^wT$, we have 
  \begin{equation*}
    \rho^wr \leq \norm{d}\leq \rho^w R
  \end{equation*}
  for each non-zero digit $d\in\cD$.
\end{remark}


Further, we get the following continuity result.


\begin{proposition}\label{pro:value-continuous}
  The value function $\NAFvaluename$ is Lipschitz continuous on
  $\wNAFsetfininf$.
\end{proposition}


This result is a consequence of the boundedness of the digit set,
see~\cite{Heuberger-Krenn:2012:wnaf-analysis} for a formal proof.


We need the full block length distribution theorem from Heuberger and
Krenn~\cite{Heuberger-Krenn:2012:wnaf-analysis}. This was proved for numeral
systems with algebraic integer $\tau$ as base. But the result does not depend
on $\tau$ directly, only on the size of the digit set, which depends on the
norm of $\tau$. In our case this norm equals $\rho^n$. That replacement is
already done in the theorem written down below.


\begin{theorem}[Full Block Length Distribution Theorem]
  \label{th:w-naf-distribution}
  Denote the number of \wNAF{}s of length $m\in\N_0$ by $C_m$. We get
  \begin{equation*}
    C_m = \frac{1}{(\rho^n-1)w+1}\rho^{n(m+w)}+\Oh{(\mu\rho^n)^m},
  \end{equation*}
  where $\mu=(1+\frac{1}{\rho^n w^3})^{-1}<1$.

  Further let $0\neq \eta\in\cD$ be a fixed digit and define the random
  variable $X_{m,\eta}$ to be the number of occurrences of the digit $\eta$ in
  a random \wNAF{} of length $m$, where every \wNAF{} of length $m$ is assumed
  to be equally likely. Then we get
  \begin{equation*}
     \expect{X_{m,\eta}} = Em + \Oh{1}
  \end{equation*}
  for the expectation, where
  \begin{equation*}
    E = \frac{1}{\rho^{n(w-1)}((\rho^n-1)w+1)}.    
  \end{equation*}
\end{theorem}


The theorem in~\cite{Heuberger-Krenn:2012:wnaf-analysis} gives more details,
which we do not need for the results in this article: We have
\begin{equation*}
  \expect{X_{m,\eta}} = Em  + E_0 + \Oh{m\mu^m}
\end{equation*}
with an explicit constant term $E_0$. Further the variance
\begin{equation*}
  \variance{X_{n,w,\eta}} = Vm + V_0 + \Oh{m^2\mu^m}
\end{equation*}
with explicit constants $V$ and $V_0$ is calculated, and a central limit
theorem is proved.


\section{Bounds for the Value of Non-Adjacent Forms}
\label{sec:bounds-value}


In this section we have a closer look at the value of a \wNAF{}. We want to
find upper bounds, as well as a lower bound for it. In the proofs of all those
bounds we use bounds for the norm $\norm{\,\cdot\,}$. More precisely, geometric
parameters of the tiling $T$, i.e., the already defined reals $r$ and $R$, are
used.

The following proposition deals with three upper bounds, one for
the norm of the value of a \wNAF{}-expansion and two give us bounds in
conjunction with the tiling.


\begin{proposition}[Upper Bounds]\label{pro:upper-bound-fracnafs}
  Let $\bfeta\in\wNAFsetfininf$, and denote the position of the most significant
  digit of $\bfeta$ by $J$. Let
  \begin{equation*}
    B_U = \frac{\rho^wR}{1 - \rho^{-w}}.
  \end{equation*}
  Then the following statements are true:
 
  \begin{enumerate}[(a)]

  \item \label{enu:upper-bound:leq-f} We get
    \begin{equation*}
      \norm{\NAFvalue{\bfeta}} 
      \leq \rho^J B_U.
    \end{equation*}
  
  \item \label{enu:upper-bound:in-balls} We have
    \begin{equation*}
      \NAFvalue{\bfeta} \in 
      \bigcup_{z \in \Phi^{w+J} T} \ball*{z}{\rho^{-w+J} B_U}.
    \end{equation*}

  \item \label{enu:upper-bound:in-v} We get
    \begin{equation*}
      \NAFvalue{\bfeta} \in \Phi^{2w+J} T.
    \end{equation*}

  \item \label{enu:upper-bound:approx-in-v} For each $\ell\in\N_0$, we have
    \begin{equation*}
      \NAFvalue{0\bfldot\eta_{-1}\ldots\eta_{-\ell}} + \Phi^{-\ell}T 
      \subseteq \Phi^{2w-1} T.
    \end{equation*}
  \end{enumerate}
\end{proposition}

Note that $\rho^J = \NAFd{\bfeta}{\bfzero}$, so we can rewrite the statements
of the proposition above in terms of that metric, see also
Corollary~\ref{cor:bounds-value}.


\begin{proof}
  \begin{enumerate}[(a)]

  \item In the calculations below, we use the Iversonian notation
    $\iverson{\ifdissertationkrenn\mathrm{expr}\else\var{expr}\fi}=1$ if
    $\ifdissertationkrenn\mathrm{expr}\else\var{expr}\fi$ is true and
    $\iverson{\ifdissertationkrenn\mathrm{expr}\else\var{expr}\fi}=0$
    otherwise, cf.\ Graham, Knuth and
    Patashnik~\cite{Graham-Knuth-Patashnik:1994}.
    
    The result follows trivially for $\bfeta=\bfzero$. First assume that the
    most significant digit of $\bfeta$ is at position $0$. Since
    $\norm{\eta_{-j}} \leq \rho^wR$ (cf.\ Remark~\ref{rem:digitset:tiling}),
    $\rho>1$ and $\bfeta$ is fulfilling the \wNAF{}-condition, we obtain
    \begin{align*}
      \norm{\NAFvalue{\bfeta}}
      &= \norm{\sum_{j=0}^{\infty} \Phi^{-j}\eta_{-j} }
      \leq \sum_{j=0}^{\infty} \norm{\Phi^{-1}}^j \norm{\eta_{-j}}
      = \sum_{j=0}^{\infty} \rho^{-j} \norm{\eta_{-j}} \\
      &\leq \rho^w R \sum_{j=0}^{\infty} \rho^{-j} 
      \tiverson{\eta_{-j} \neq 0} 
      \leq \rho^w R \sum_{j=0}^{\infty} \rho^{-j} 
      \tiverson{-j \equiv 0 \pmod{w}} \\
      &= \rho^w R \sum_{k=0}^{\infty} \rho^{-wk}
      = \frac{\rho^w R}{1-\rho^{-w}} = B_U.
    \end{align*}    

    In the general case, we have the most significant digit of $\bfeta$ at a
    position $J$. We get $\NAFvalue{\bfeta} = \Phi^J\NAFvalue{\bfeta'}$ for a
    \wNAF{} $\bfeta'$ with most significant digit at position $0$. Therefore we
    obtain
    \begin{equation*}
      \norm{\NAFvalue{\bfeta}} 
      = \norm{\Phi^J\NAFvalue{\bfeta'}}
      \leq \norm{\Phi}^J \norm{\NAFvalue{\bfeta'}}
      \leq \rho^J B_U, 
    \end{equation*}
    which was to be proved.

  \item There is nothing to show if the \wNAF{} $\bfeta$ is
    zero. First suppose that the most significant digit is at position
    $w$. Then, using \itemref{enu:upper-bound:leq-f}, we have
    \begin{equation*}
      \norm{\NAFvalue{\bfeta} - \Phi^w\eta_w} \leq B_U,
    \end{equation*}
    therefore
    \begin{equation*}
      \NAFvalue{\bfeta} \in \ball*{\Phi^w\eta_w}{B_U}.
    \end{equation*}
    Since $\eta_w \in \Phi^w T$, the statement follows for the special
    case. The general case is again obtained by shifting.

  \item Using the upper bound found in \itemref{enu:upper-bound:leq-f} and the
    assumption~\eqref{eq:existence-condition} yields
    \begin{equation*}
      \norm{\NAFvalue{\bfeta}} \leq \rho^J B_U  
      = \rho^J \frac{\rho^w R}{1 - \rho^{-w}} 
      \leq r\rho^{2w+J}.
    \end{equation*}
    Since $\ball*{0}{r\rho^{2w+J}} \subseteq \Phi^{2w+J} T$, the statement
    follows.
   
  \item Analogously to the proof of~\itemref{enu:upper-bound:leq-f}, except
    that we use $\ell$ for the upper bound of the sum, we obtain for $v\in T$
    \begin{align*}
      \norm{\NAFvalue{0\bfldot\eta_{-1}\ldots\eta_{-\ell}} + \Phi^{-\ell}v}
      &\leq \norm{\NAFvalue{0\bfldot\eta_{-1}\ldots\eta_{-\ell}}} 
      + \rho^{-\ell} R \\
      &\leq \rho^{-1} \frac{\rho^w R}{1 - \rho^{-w}}
      \left( 1 - \rho^{-w \floor{\frac{\ell-1+w}{w}}} \right)
      + \rho^{-\ell} R \\
      &\leq \frac{\rho^{w-1} R}{1 - \rho^{-w}}
      \left( 1 - \rho^{-\ell+1-w} 
        + \rho^{-\ell+1-w} \left(1 - \rho^{-w}\right) \right) \\
      &= \frac{\rho^{w-1} R}{1 - \rho^{-w}}
      \left( 1 - \rho^{-\ell+1-2w} \right).
    \end{align*}
    Since $1 - \rho^{-\ell+1-2w} < 1$, we get
    \begin{equation*}
      \norm{\NAFvalue{0\bfldot\eta_{-1}\ldots\eta_{-\ell}} + \Phi^{-\ell}T}
      \leq \rho^{-1} \frac{\rho^w R}{1 - \rho^{-w}} = \rho^{-1} B_U
    \end{equation*}
    for all $\ell\in\N_0$. By the same argumentation as in the proof
    of~\itemref{enu:upper-bound:in-v}, the statement follows. \qedhere
  \end{enumerate}
\end{proof}


Next we want to find a lower bound for the value of a \wNAF{}. Clearly the
\wNAF{} $\bfzero$ has value $0$, so we are interested in cases where we have a
non-zero digit somewhere.


\begin{proposition}[Lower Bound]\label{pro:lower-bound-fracnafs}
  Let $\bfeta\in\wNAFsetfininf$ be non-zero, and denote the position of the
  most significant digit of $\bfeta$ by $J$. Then we have
  \begin{equation*}
    \norm{\NAFvalue{\bfeta}} \geq \rho^J B_L,
  \end{equation*}
  where
  \begin{equation*}
    B_L = r - \rho^{-2w} B_U = r - \frac{R}{\rho^w - 1}.
  \end{equation*}
\end{proposition}

Note that $B_L>0$ is equivalent to $\frac{R}{r} < \rho^w-1$, i.e.\ the
assumption~\eqref{eq:existence-condition}. Moreover, we have
\begin{equation*}
  \frac{R}{r-B_L} = \rho^w-1.
\end{equation*}


\begin{proof}[Proof of Proposition~\ref{pro:lower-bound-fracnafs}]
  First suppose the most significant digit of the \wNAF{}~$\bfeta$ is at
  position $0$ and the second non-zero digit (read from left to right) at
  position~$J$. Then
  \begin{equation*}
    \NAFvalue{\bfeta} - \eta_0
    = \sum_{k=w}^\infty \Phi^{-k}\eta_{-k} 
    \in \bigcup_{z \in T} \ball*{z}{\rho^{-w+J} B_U}
    \subseteq \bigcup_{z \in T} \ball*{z}{\rho^{-2w} B_U}
  \end{equation*}
  according to \itemref{enu:upper-bound:in-balls} of
  Proposition~\ref{pro:upper-bound-fracnafs}. Therefore
  \begin{equation*}
    \NAFvalue{\bfeta} \in \bigcup_{z \in T_{\eta_0}} \ball*{z}{\rho^{-2w} B_U}.
  \end{equation*}
  This means that $\NAFvalue{\bfeta}$ is in $T_{\eta_0}$ or in a $\rho^{-2w}
  B_U$\nbd-strip around this cell.  The two tiling cells $T_{\eta_0}$ for
  $\eta_0$ and $T_0=T$ for $0$ are disjoint, except for parts of the boundary,
  if they are adjacent. Since a ball with radius $r$ is contained in each
  tiling cell, we deduce that
  \begin{equation*}
    \norm{\NAFvalue{\bfeta}}
    \geq r - \rho^{-2w} B_U
    = r - \frac{R}{\rho^w - 1} = B_L,
  \end{equation*}
  which was to be shown. The case of a general $J$ is again, as in the proof of
  Proposition~\ref{pro:upper-bound-fracnafs}, obtained by shifting.
\end{proof}


Combining the previous two propositions leads to the following corollary, which
gives an upper and a lower bound for the norm of the value of a \wNAF{} by
looking at the largest non-zero index.


\begin{corollary}[Bounds for the Value]\label{cor:bounds-value}
  Let $\bfeta\in\wNAFsetfininf$, then we get
  \begin{equation*}
     \NAFd{\bfeta}{\bfzero} B_L
    \leq \norm{\NAFvalue{\bfeta}}
    \leq \NAFd{\bfeta}{\bfzero} B_U.
  \end{equation*}
\end{corollary}


\begin{proof}
  This follows directly from Propositions~\ref{pro:upper-bound-fracnafs}
  and~\ref{pro:lower-bound-fracnafs}, since the term $\rho^J$ is
  equal to $\NAFd{\bfeta}{\bfzero}$.
\end{proof}


Lastly in this section, we want to find out if there are special \wNAF{}s for
which we know for sure that all their expansions start with a certain finite
\wNAF{}. This is formulated in the following lemma.


\begin{lemma}\label{lem:choosing-k0}
  There is a $k_0\in\N_0$ such that for all $k\geq k_0$ the following
  holds: If $\bfeta\in\wNAFsetinf$ starts with the word $0^k$, i.e.,
  $\eta_{-1}=0$, \ldots, $\eta_{-k}=0$, then we get for all
  $\bfxi\in\wNAFsetfininf$ that $\NAFvalue{\bfxi}=\NAFvalue{\bfeta}$ implies
  $\bfxi\in\wNAFsetellinf{0}$.
\end{lemma}


\begin{proof}
  Let $\bfxi = \bfxi_I\bfldot\bfxi_F$. Then
  $\norm{\NAFvalue{\bfxi_I\bfldot\bfxi_F}} < B_L$ implies $\bfxi_I=\bfzero$,
  cf.\ Corollary~\ref{cor:bounds-value}. Further, for our $\bfeta$ we obtain
  $z=\norm{\NAFvalue{\bfeta}} \leq \rho^{-k} B_U$. So it is sufficient to show
  that
  \begin{equation*}
    \rho^{-k} B_U < B_L,
  \end{equation*} 
  which is equivalent to
  \begin{equation*}
    k > \log_\rho \frac{B_U}{B_L}.
  \end{equation*}
  We obtain
  \begin{equation*}
    k > 2w - \log_\rho \left(\frac{r}{R}\left(\rho^w-1\right)-1\right),
  \end{equation*}
  where we just inserted the formulas for $B_U$ and $B_L$. Choosing an
  appropriate $k_0$ is now easily possible.
\end{proof}

Note that we can find a constant $k_1$ independent from $w$ such that for
all $k\geq 2w+k_1$ the assertion of Lemma~\ref{lem:choosing-k0} holds. This can
be seen in the proof, since $\frac{r}{R}\left(\rho^w-1\right)-1$ is
monotonically increasing in $w$.


\section{Right-infinite Expansions}
\label{sec:right-inf-expansions}


We have the existence of a (finite integer) \wNAF{}-expansion for each element
of the lattice $\Lambda\subseteq\R^n$, cf.\
Remark~\ref{rem:digitset:tiling}. But that existence condition is also
sufficient to get \wNAF{}-expansions for all elements in~$\R^n$. Those
expansions possibly have an infinite right-length. The aim of this section is
to show that result. The proofs themselves are a minor generalisation of the
ones given in~\cite{Heuberger-Krenn:2012:wnaf-analysis} for the quadratic case.


We will use the following abbreviation in this section. We define
\begin{equation*}
  [\Phi^{-1}]\Lambda \colonequals \bigcup_{j\in\N_0} \Phi^{-j}\Lambda.
\end{equation*}
Note that $\Lambda \subseteq \Phi^{-1}\Lambda$.


To prove the existence theorem of this section, we need the following three
lemmata.


\begin{lemma}\label{lem:value-injective}
  The function $\NAFvaluename\restricted{\wNAFsetfinfin}$ is injective.
\end{lemma}


\begin{proof}
  Let $\bfeta$ and $\bfxi$ be elements of $\wNAFsetfinfin$ with
  $\NAFvalue{\bfeta} = \NAFvalue{\bfxi}$. This implies that $\Phi^J
  \NAFvalue{\bfeta} = \Phi^J \NAFvalue{\bfxi} \in \Lambda$ for some
  $J\in\Z$. By uniqueness of the integer \wNAF{}s we conclude that
  $\bfeta=\bfxi$.
\end{proof}


\begin{lemma}\label{lem:Zinvtau-is-wnaffinfin}
  We have $\NAFvalue{\wNAFsetfinfin} = [\Phi^{-1}]\Lambda$.
\end{lemma}


\begin{proof}
  Let $\bfeta\in\wNAFsetfinfin$. There are only finitely many $\eta_j\neq 0$,
  so there is a $J\in\N_0$ such that
  $\NAFvalue{\bfeta}\in\Phi^{-J}\Lambda$. Conversely, if
  $z\in\Phi^{-J}\Lambda$, then there is an integer \wNAF{} of $\Phi^Jz$, and
  therefore, there is a $\bfxi\in\wNAFsetfinfin$ with $\NAFvalue{\bfxi}=z$.
\end{proof}


\begin{lemma}\label{lem:Zinvtau-dense}
  $[\Phi^{-1}]\Lambda$ is dense in $\R^n$.
\end{lemma}


\begin{proof}
  Let $\Lambda = w_1\Z \oplus \dots \oplus w_n\Z$ for linearly independent
  $w_1,\dots,w_n\in\R^n$. Let $z\in\R^n$ and $K\in\N_0$. Then $\Phi^K z=
  z_1w_1+\dots+z_nw_n$ for some reals $z_1,\dots,z_n$. We have
  \begin{equation*}
    \norm{z-\left(\floor{z_1}\Phi^{-K}w_1+\dots+\floor{z_n}\Phi^{-K}w_n\right)}
    < \rho^{-K}\left(\norm{w_1}+\dots+\norm{w_n}\right),
  \end{equation*}
  which proves the lemma.
\end{proof}


Now we can prove the following theorem.


\begin{theorem}[Existence Theorem concerning $\R^n$]\label{thm:C-has-wnaf-exp}
  Let $z\in\R^n$. Then there is an $\bfeta\in\wNAFsetfininf$ such that
  $z=\NAFvalue{\bfeta}$, i.e., each element in $\R^n$ has a \wNAF{}-expansion.
\end{theorem}


\begin{proof}
  By Lemma~\ref{lem:Zinvtau-dense}, there is a sequence
  $z_n\in[\Phi^{-1}]\Lambda$ converging to $z$. By
  Lemma~\ref{lem:Zinvtau-is-wnaffinfin}, there is a sequence
  $\bfeta_n\in\wNAFsetfinfin$ with $\NAFvalue{\bfeta_n}=z_n$ for all $n$. By
  Corollary~\ref{cor:bounds-value} the sequence $\NAFd{\bfeta_n}{0}$ is
  bounded from above, so there is an $\ell$ such that
  $\bfeta_n\in\wNAFsetellfin{\ell} \subseteq \wNAFsetellinf{\ell}$. By
  Proposition~\vref{pro:naf-compact}, we conclude that there is a convergent
  subsequence $\bfeta'_n$ of $\bfeta_n$. Set $\bfeta\colonequals
  \lim_{n\to\infty}{\bfeta'_n}$. By continuity of $\NAFvaluename$, see
  Proposition~\vref{pro:value-continuous}, we conclude that
  $\NAFvalue{\bfeta}=z$.
\end{proof}


\section{The Fundamental Domain}
\label{sec:fundamental-domain}


We now derive properties of the \emph{Fundamental Domain}, i.e., the subset
of~$\R^n$ representable by \wNAF{}s which vanish left of the
$\Phi$\nbd-point. The boundary of the fundamental domain is shown to correspond
to elements which admit more than one \wNAF{}s differing left of the
$\Phi$\nbd-point. Finally, an upper bound for the Hausdorff dimension of the
boundary is derived.

All the results in this section are generalisations of the propositions and
remarks found in~\cite{Heuberger-Krenn:2012:wnaf-analysis}.  For some of those
results given here, the proof is the same as in the quadratic case or a
straightforward generalisation of it. In those cases the proofs will be
skipped.

We start with the formal definition of the fundamental domain.


\begin{definition}[Fundamental Domain]\label{def:fund-domain}
  The set
  \begin{equation*}
    \cF \colonequals  \NAFvalue{\wNAFsetinf} 
    = \set*{\NAFvalue{\bfxi}}{\bfxi\in\wNAFsetinf}
  \end{equation*}
  is called \emph{fundamental domain}.
\end{definition}


The pictures in Figure~\vref{fig:w-weta} show some fundamental domains for
lattices coming from imaginary-quadratic algebraic integers~$\tau$. We continue
with some properties of fundamental domains. We have the following compactness
result.


\begin{proposition}\label{pro:fund-domain-compact}
  The fundamental domain $\cF$ is compact.
\end{proposition}

\begin{proof}
  The proof is a straightforward generalisation of the proof of the quadratic
  case in~\cite{Heuberger-Krenn:2012:wnaf-analysis}.
\end{proof}


We can also compute the Lebesgue measure of the fundamental domain. This result
can be found in Remark~\vref{rem:meas-fund-domain}. To calculate~$\lmeas{\cF}$,
we will need the results of Sections~\ref{sec:cell-rounding-op}
and~\ref{sec:sets-w_eta}.


The space~$\R^n$ has a tiling property with respect to the fundamental
domain. This fact is stated in the following proposition.


\begin{proposition}[Tiling Property]\label{cor:complex-plane-tiling}
  The space~$\R^n$ can be tiled with scaled versions of the fundamental domain
  $\cF$. Only finitely many different sizes are needed. More precisely: Let
  $K\in\Z$, then
  \begin{equation*}
    \R^n = 
    \bigcup_{\substack{k\in\set{K,K+1,\dots,K+w-1} 
        \\ \bfxi \in \wNAFsetfinell{0}
        \\ \text{$k\neq K+w-1$ implies $\bfxi_0\neq0$}}}
    \left( \Phi^k \NAFvalue{\bfxi} + \Phi^{k-w+1} \cF \right),
  \end{equation*}
  and the intersection of two different $\Phi^k \NAFvalue{\bfxi} + \Phi^{k-w+1}
  \cF$ in this union is a subset of the intersection of their boundaries.
\end{proposition}


\begin{proof}
  The proof is a straightforward generalisation of the proof of the quadratic
  case in~\cite{Heuberger-Krenn:2012:wnaf-analysis}.
\end{proof}


Note that the intersection of the two different sets of the tiling in the
previous corollary has Lebesgue measure $0$. This will be a consequence of
Proposition~\ref{pro:boundary-fund-dom-dim-upper}.


\begin{remark}[Iterated Function System]\label{rem:ifs}
  Define $\f{f_0}{z} = \Phi^{-1}z$ and for a non-zero digit
  $\vartheta\in\cD^\bullet$ define $\f{f_\vartheta}{z} = \Phi^{-1}\vartheta +
  \Phi^{-w}z$. Then the \emph{(affine) iterated function system}
  $\ifs{f_\vartheta}{\vartheta\in\cD}$, cf.\ Edgar~\cite{Edgar:2008:measur} or
  Barnsley~\cite{Barnsley:1988:fractals-ew}, has the fundamental domain~$\cF$
  as an invariant set, i.e., 
  \begin{equation*}
    \cF = \bigcup_{\vartheta\in\cD} \f{f_\vartheta}{\cF}
    = \Phi^{-1}\cF \cup \bigcup_{\vartheta\in\cD^\bullet} 
    \left( \Phi^{-1}\vartheta + \Phi^{-w}\cF \right).
  \end{equation*}
  That formula also reflects the fact that we have two possibilities building
  the elements $\bfxi\in\wNAFsetinf$ from left to right: We can either append
  $0$, which corresponds to an application of $\Phi^{-1}$, or we can append a
  non-zero digit $\vartheta\in\cD^\bullet$ and then add $w-1$ zeros.

  Furthermore, the iterated function system
  $\ifs{f_\vartheta}{\vartheta\in\cD}$ fulfils \emph{Moran's open set
    condition}\footnote{``Moran's open set condition'' is sometimes just called
    ``open set condition''}, cf.\ Edgar~\cite{Edgar:2008:measur} or
  Barnsley~\cite{Barnsley:1988:fractals-ew}. The \emph{Moran open set} used is
  $\interior*{\cF}$. This set satisfies
  \begin{equation*}
    \f{f_\vartheta}{\interior*{\cF}}\cap\f{f_{\vartheta'}}{\interior*{\cF}}
    =\emptyset
  \end{equation*}
  for $\vartheta \neq \vartheta'\in\cD$ and
  \begin{equation*}
    \interior*{\cF} \supseteq \f{f_\vartheta}{\interior*{\cF}}
  \end{equation*}
  for all $\vartheta\in\cD$. We remark that the first condition follows directly
  from the tiling property in Corollary~\ref{cor:complex-plane-tiling} with
  $K=-1$. The second condition follows from the fact that $f_\vartheta$ is an
  open mapping.
\end{remark}


Next we want to have a look at the Hausdorff dimension of the boundary of
$\cF$. We will need the following characterisation of the boundary.


\begin{proposition}[Characterisation of the Boundary]\label{pro:char-boundary}
  Let $z\in\cF$. Then $z\in\boundary*{\cF}$ if and only if there exists a
  \wNAF{} $\bfxi_I\bfldot\bfxi_F\in\wNAFsetfininf$ with $\bfxi_I\neq\bfzero$
  such that $z=\NAFvalue{\bfxi_I\bfldot\bfxi_F}$.
\end{proposition}


\begin{proof}
  The proof is a straightforward generalisation of the proof of the quadratic
  case in~\cite{Heuberger-Krenn:2012:wnaf-analysis}.
\end{proof}


The following proposition deals with the Hausdorff dimension of the boundary of
$\cF$.


\begin{proposition}\label{pro:boundary-fund-dom-dim-upper}
  For the Hausdorff dimension of the boundary of the fundamental domain we get
  $\hddim \boundary*{\cF} < n$.
\end{proposition}


The idea of this proof is similar to a proof in Heuberger and
Prodinger~\cite{Heuberger-Prodinger:2006:analy-alter}, and it is a
generalisation of the one given in~\cite{Heuberger-Krenn:2012:wnaf-analysis}.


\begin{proof}
  Set $k\colonequals k_0+w-1$ with $k_0$ from Lemma~\vref{lem:choosing-k0}. For
  $j\in\N$ define
  \begin{equation*}
    U_j \colonequals  \set*{\bfxi\in\wNAFsetellell{0}{j}}{
      \text{$\xi_{-\ell}\xi_{-(\ell+1)}\ldots\xi_{-(\ell+k-1)} \neq 0^k$ 
        for all $\ell\in\set{1,\dots,j-k+1}$}}.
  \end{equation*}
  The elements of $U_j$, or more precisely the digits from index~$-1$ to~$-j$,
  can be described by the regular expression
  \begin{equation*}
    \left( \eps + \sum_{d\in\cD^\bullet} \sum_{\ell=0}^{w-2} 0^\ell d \right)
    \left( \sum_{d\in\cD^\bullet} \sum_{\ell=w-1}^{k-1} 0^\ell d \right)^\bfast
    \left( \sum_{\ell=0}^{k-1} 0^\ell \right).
  \end{equation*}
  This can be translated to the generating function
  \begin{equation*}
    \f{G}{Z} = \sum_{j\in\N} \card*{U_j} Z^j
    = \left( 1 + \card*{\cD^\bullet} \sum_{\ell=0}^{w-2} Z^{\ell+1} \right)
    \frac{1}{1 - \card*{\cD^\bullet} \sum_{\ell=w-1}^{k-1} Z^{\ell+1}}
    \left( \sum_{\ell=0}^{k-1} Z^\ell \right)
  \end{equation*}
  used for counting the number of elements in $U_j$. Rewriting yields
  \begin{equation*}
    \f{G}{Z} = \frac{1-Z^k}{1-Z} 
    \frac{1 + (\card*{\cD^\bullet} - 1) Z - \card*{\cD^\bullet} Z^w}{
      1 - Z - \card*{\cD^\bullet} Z^w + \card*{\cD^\bullet} Z^{k+1}},
  \end{equation*}
  and we set
  \begin{equation*}
    \f{q}{Z} \colonequals  
    1 - Z - \card*{\cD^\bullet} Z^w + \card*{\cD^\bullet} Z^{k+1}.
  \end{equation*}

  \begin{figure}
    \centering
    \includegraphics{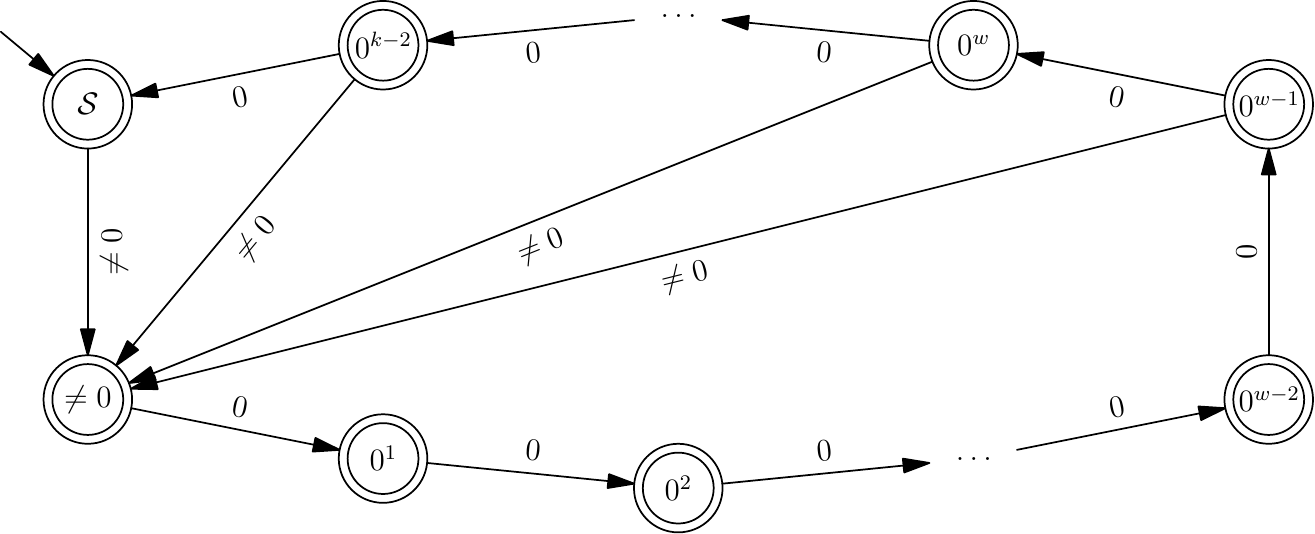}
    \caption[Automaton recognising $\bigcup_{j\in\N} \wt{U}_j$]{Automaton $\cA$
      recognising $\bigcup_{j\in\N} \wt{U}_j$ from right to left, see proof of
      Proposition~\ref{pro:boundary-fund-dom-dim-upper}. The state $\cS$ is the
      starting state, all states are valid end states. An edge marked with
      $\neq0$ means one edge for each non-zero digit in the digit set $\cD$. The
      state $\neq0$ means that there was an non-zero digit read, a state
      $0^\ell$ means that $\ell$ zeros have been read.}
    \label{fig:uj-automata}
  \end{figure}
  
  Now we define
  \begin{equation*}
    \wt{U}_j \colonequals  \set*{\bfxi \in U_j}{\xi_{-j} \neq 0}
  \end{equation*}
  and consider $\wt{U} \colonequals \bigcup_{j\in\N} \wt{U}_j$. Suppose
  $w\geq2$. The \wNAF{}s in that set, or more precisely the finite strings from
  index $-1$ to the smallest index of a non-zero digit, will be recognised by
  the automaton $\cA$ which is shown in Figure~\ref{fig:uj-automata} and reads
  its input from right to left. It is easy to see that the underlying directed
  graph $G_\cA$ of the automaton $\cA$ is strongly connected, therefore its
  adjacency matrix $M_\cA$ is irreducible. Since there are cycles of length $w$
  and $w+1$ in the graph and $\f{\gcd}{w,w+1}=1$, the adjacency matrix is
  primitive. Thus, using the Perron-Frobenius theorem we obtain
  \begin{align*}
    \card*{\wt{U}_j} &= \card{\text{walks in $G_\cA$ of length $j$ from starting
        state $\cS$ to some other state}} \\
    &= \begin{pmatrix} 1 & 0 & \dots & 0 \end{pmatrix}
    M_\cA^j
    \begin{pmatrix} 1 \\ \vdots \\ 1 \end{pmatrix}
    = \wt{c} \left(\sigma \rho^n \right)^j \left(1+\Oh{s^j}\right)
  \end{align*}
  for a $\wt{c}>0$, a $\sigma>0$, and an $s$ with $0 \leq s < 1$. Since the
  number of \wNAF{}s of length $j$ is $\Oh{\rho^{nj}}$, see
  Theorem~\vref{th:w-naf-distribution}, we get $\sigma\leq1$.
  
  We clearly have
  \begin{equation*}
    U_j = \biguplus_{\ell=j-k+1}^j \wt{U}_\ell,
  \end{equation*}
  so we get 
  \begin{equation*}
    \card*{U_j} = \coefficient{Z^j} \f{G}{Z}
    = c \left(\sigma\rho^n\right)^j \left( 1 + \Oh{s^j} \right)
  \end{equation*}
  for some constant $c>0$.
 
  To rule out $\sigma=1$, we insert the ``zero'' $\rho^{-n}$ in
  $\f{q}{Z}$. We obtain
  \begin{align*}
    \f{q}{\rho^{-n}} 
    &= 1 - \rho^{-n} - \card*{\cD^\bullet} \rho^{-nw} 
    + \card*{\cD^\bullet} \rho^{-n(k+1)} \\
    &= 1 - \rho^{-n} 
    - \rho^{n(w-1)}\left(\rho^n-1\right) \rho^{-nw} 
    + \rho^{n(w-1)}\left(\rho^n-1\right) \rho^{-n(k+1)} \\
    &= \left(\rho^n-1\right) \rho^{n(w-k-2)} > 0,
  \end{align*}
  where we used the cardinality of $\cD^\bullet$ from our set-up in
  Section~\ref{sec:set-up} and $\rho>1$. Therefore we get $\sigma<1$. It is
  easy to check that the result for $\card*{U_j}$ holds in the case $w=1$,
  too.

  Define
  \begin{equation*}
    U \colonequals  \set*{\NAFvalue{\bfxi}}{
      \text{$\bfxi\in\wNAFsetinf$ with 
        $\xi_{-\ell}\xi_{-(\ell+1)}\ldots\xi_{-(\ell+k-1)} \neq 0^k$ 
        for all $\ell \geq 1$}}.
  \end{equation*}
  We want to cover $U$ with hypercubes. Let $C\subseteq\R^n$ be the closed
  paraxial hypercube with centre $0$ and width $2$. Using
  Proposition~\vref{pro:upper-bound-fracnafs} yields
  \begin{equation*}
    U \subseteq 
    \bigcup_{z\in\NAFvalue{U_j}} \left(z + B_U\rho^{-j} C\right)
  \end{equation*}
  for all $j\in\N$, i.e., $U$ can be covered with $\card*{U_j}$ boxes of size
  $2B_U\rho^{-j}$. Thus we get for the upper box dimension, cf.\
  Edgar~\cite{Edgar:2008:measur},
  \begin{equation*}
     \uboxdim U \leq \lim_{j\to\infty} 
     \frac{\log \card*{U_j}}{- \log (2B_U\rho^{-j})}.
  \end{equation*}
  Inserting the cardinality $\card*{U_j}$ from above, using the logarithm to
  base $\rho$ and $0 \leq s < 1$ yields
  \begin{equation*}
    \uboxdim U \leq \lim_{j\to\infty}  
    \frac{\log_\rho c + j \log_\rho (\sigma\rho^n)
      + \log_\rho (1 + \Oh{s^j}) }{j + \Oh{1}}
    = n + \log_\rho \sigma.
  \end{equation*}
  Since $\sigma<1$, we get $\uboxdim U < 2$.

  Now we will show that $\boundary*{\cF} \subseteq U$. Clearly $U \subseteq
  \cF$, so the previous inclusion is equivalent to $\cF \setminus U \subseteq
  \interior{\cF}$. So let $z \in \cF \setminus U$. Then there is a
  $\bfxi\in\wNAFsetinf$ such that $z=\NAFvalue{\bfxi}$ and $\bfxi$ has a block
  of at least $k$ zeros somewhere on the right hand side of the
  $\Phi$\nbd-point. Let $\ell$ denote the starting index of this block, i.e.,
  \begin{equation*}
    \bfxi=0\bfldot\underbrace{\xi_{-1}\ldots\xi_{-(\ell-1)}}_{=: \bfxi_A}
    0^k\xi_{-(\ell+k)}\xi_{-(\ell+k+1)}\ldots.
  \end{equation*}
  Let $\bfvartheta =
  \bfvartheta_I\bfldot\bfvartheta_A\vartheta_{-\ell}\vartheta_{-(\ell+1)}\ldots
  \in \wNAFsetfininf$ with $\NAFvalue{\bfvartheta}=z$. We have
  \begin{equation*}
    z = \NAFvalue{0\bfldot\bfxi_A} + \Phi^{-\ell-w} z_\xi
    = \NAFvalue{\bfvartheta_I\bfldot\bfvartheta_A} 
    + \Phi^{-\ell-w} z_\vartheta
  \end{equation*}
  for appropriate $z_\xi$ and $z_\vartheta$. By Lemma~\vref{lem:choosing-k0},
  all expansions of $z_\xi$ are in $\wNAFsetinf$. Thus all expansions of
  \begin{equation*}
    \NAFvalue{\bfvartheta_I\bfvartheta_A}
    + \Phi^{-(w-1)} z_\vartheta
    - \NAFvalue{\bfxi_A} 
    = \Phi^{\ell-1} z - \NAFvalue{\bfxi_A}
    = \Phi^{-(w-1)} z_\xi
  \end{equation*}
  start with $0.0^{w-1}$, since our choice of $k$ is $k_0+w-1$. As the unique
  \wNAF{} of $\NAFvalue{\bfvartheta_I\bfvartheta_A} - \NAFvalue{\bfxi_A}$
  concatenated with any \wNAF{} of $\Phi^{-(w-1)}z_\vartheta$ gives rise to
  such an expansion, we conclude that $\NAFvalue{\bfvartheta_I\bfvartheta_A} -
  \NAFvalue{\bfxi_A} = 0$ and therefore $\bfvartheta_I=\bfzero$ and
  $\bfvartheta_A=\bfxi_A$.  So we conclude that all representations of $z$ as a
  \wNAF{} have to be of the form $0\bfldot\bfxi_A0^{w-1}\bfeta$ for some
  \wNAF{} $\bfeta$. Thus, by using Proposition~\ref{pro:char-boundary}, we get
  $z\not\in\boundary*{\cF}$ and therefore $z\in\interior{\cF}$.

  Until now we have proved
  \begin{equation*}
    \uboxdim \boundary*{\cF} \leq \uboxdim U < n.
  \end{equation*}
  Because the Hausdorff dimension of a set is at most its upper box dimension,
  cf.\ Edgar~\cite{Edgar:2008:measur} again, the desired result follows.
\end{proof}


\section{Cell Rounding Operations}
\label{sec:cell-rounding-op}

\let\origvarhexagon\varhexagon
\renewcommand\varhexagon{{T}}

In this section we define operators working on subsets of the
space~$\R^n$. These will use the lattice $\Lambda$ and the tiling $T$. They
will be a very useful concept to prove Theorem~\ref{thm:countdigits}.


\begin{definition}[Cell Rounding Operations]\label{def:round-v}
  Let $B\subseteq\R^n$ and $j\in\Z$. We define the \emph{cell packing of $B$}
  (``floor $B$'')
  \begin{align*}
    \floorV{B} &\colonequals   \bigcup_{\substack{z \in \Lambda \\ T_z \subseteq  B}} T_z 
    & &\text{and} &
    \floorV[j]{B} &\colonequals  \f{\Phi^{-j}}{\floorV{\Phi^j B}},
    \intertext{the \emph{cell covering of $B$} (``ceil $B$'')}
    \ceilV{B} &\colonequals  \closure{\floorV{B^C}^C}
    & &\text{and} &
    \ceilV[j]{B} &\colonequals  \f{\Phi^{-j}}{\ceilV{\Phi^j B}}, 
    \intertext{the \emph{fractional cells of $B$}}
    \fracpartV{B} &\colonequals  B \setminus \floorV{B}
    & &\text{and} &
    \fracpartV[j]{B} &\colonequals  \f{\Phi^{-j}}{\fracpartV{\Phi^j B}},
    \intertext{the \emph{cell covering of the boundary of $B$}}
    \boundaryV{B} &\colonequals  \closure{\ceilV{B} \setminus \floorV{B}}
    & &\text{and} &
    \boundaryV[j]{B} &\colonequals  \f{\Phi^{-j}}{\boundaryV{\Phi^j B}},
    \intertext{the \emph{cell covering of the lattice points inside $B$}}
    \coverV{B} &\colonequals  \bigcup_{\substack{z\in B \cap \Lambda}} T_z
    & &\text{and} &
    \coverV[j]{B} &\colonequals  \f{\Phi^{-j}}{\coverV{\Phi^j B}},
    \intertext{and the \emph{number of lattice points inside $B$} as}
    \cardV{B} &\colonequals  \card{B \cap \Lambda}
    & &\text{and} &
    \cardV[j]{B} &\colonequals  \cardV{\Phi^j B}.
  \end{align*}
\end{definition}


For the cell covering of a set $B$ an alternative, perhaps more intuitive
description can be given by
\begin{equation*}
  \ceilV{B} \colonequals   
  \bigcup_{\substack{z \in \Lambda \\ T_z \cap B \neq \emptyset}} T_z.
\end{equation*}


The following proposition deals with some basic properties that will be
helpful when working with those operators.


\begin{proposition}[Basic Properties of Cell Rounding Operations]
  \label{pro:round-v-basic-prop}

  Let $B\subseteq\R^n$ and $j\in\Z$.
  \begin{enumerate}[(a)]
    
  \item \label{enu:round-v-basic-prop:incl}
    We have the inclusions
      \begin{equation*}
        \floorV[j]{B} \subseteq B \subseteq \closure{B} \subseteq \ceilV[j]{B}
      \end{equation*}
      and
      \begin{equation*}
        \floorV[j]{B} \subseteq \coverV[j]{B} \subseteq \ceilV[j]{B}.    
      \end{equation*}
    For $B' \subseteq \R^n$ with $B \subseteq B'$ we get $\floorV[j]{B}
    \subseteq \floorV[j]{B'}$, $\coverV[j]{B} \subseteq \coverV[j]{B'}$ and
    $\ceilV[j]{B} \subseteq \ceilV[j]{B'}$, i.e., monotonicity with respect to
    inclusion.

  \item \label{enu:round-v-basic-prop:frac-boundary}
    The inclusion
    \begin{equation*}
      \fracpartV[j]{B} \subseteq \boundaryV[j]{B}
    \end{equation*}
    holds.

  \item \label{enu:round-v-basic-prop:boundary} We have $\boundary*{B}
    \subseteq \boundaryV[j]{B}$ and for each cell $T'$ in $\boundaryV[j]{B}$ we
    have $T' \cap \boundary*{B} \neq \emptyset$.
    
  \item \label{enu:round-v-basic-prop:card} 
    For $B' \subseteq \R^n$ with $B'$ disjoint from $B$, we get
    \begin{equation*}
      \cardV[j]{B \cup B'} = \cardV[j]{B} + \cardV[j]{B'},
    \end{equation*}
    and therefore the number of lattice points operation is monotonic with
    respect to inclusion, i.e., for $B'' \subseteq \R^n$ with $B'' \subseteq B$
    we have $\cardV[j]{B''} \leq \cardV[j]{B}$. Further we get
    \begin{equation*}
      \cardV[j]{B} 
      = \cardV[j]{\coverV[j]{B}}
      = \abs{\det\Phi}^j \frac{\lmeas{\coverV[j]{B}}}{d_\Lambda}.
    \end{equation*}
    
  \end{enumerate}
\end{proposition}


\begin{proof}
  The proof is a straightforward generalisation of the proof for
  Voronoi-tilings in the quadratic case in~\cite{Heuberger-Krenn:2012:wnaf-analysis}.  
\end{proof}


We will need some more properties concerning cardinality. We want to know the
number of points inside a region after using one of the operators. Especially we
are interested in the asymptotic behaviour, i.e., if our region becomes scaled
very large. The following proposition provides information about that.
 

\begin{proposition}\label{pro:set-nu}
  Let $U \subseteq \R^n$ bounded, measurable, and such that
  \begin{equation*}
    \cardV{\boundaryV{\Psi U}} = \Oh{\abs{\det\Psi}^{\delta/n}}
  \end{equation*}
  for $\abs{\det\Psi}\to\infty$ with maps $\Psi\colon\R^n\to\R^n$ and a fixed
  $\delta\in\R$ with $\delta>0$.

  \begin{enumerate}[(a)]
    
  \item \label{enu:set-nu:floor-ceil-cover}
    We get that each of $\cardV{\floorV{\Psi U}}$, $\cardV{\ceilV{\Psi U}}$,
    $\cardV{\coverV{\Psi U}}$ and $\cardV{\Psi U}$ equals
    \begin{equation*}
      \abs{\det\Psi} \frac{\lmeas{U}}{d_\Lambda} + \Oh{\abs{\det\Psi}^{\delta/n}}.
    \end{equation*}
    In particular, let $N\in\R$, $N>0$, and set $\Psi =
    \f{\operatorname{diag}}{N,\dots,N}$, which we identify with $N$. Then we
    get that each one of $\cardV{\floorV{NU}}$, $\cardV{\ceilV{NU}}$,
    $\cardV{\coverV{NU}}$ and $\cardV{NU}$ equals
    \begin{equation*}
      N^n \frac{\lmeas{U}}{d_\Lambda} + \Oh{N^\delta}.
    \end{equation*}
    
  \item \label{enu:set-nu:difference} Let $N\in\R$, $N>0$, and set $\Psi =
    \f{\operatorname{diag}}{N,\dots,N}$, which we identify with $N$. Then we
    get
    \begin{equation*}
      \cardV{(N+1)U \setminus NU} = \Oh{N^\delta}.
    \end{equation*}

  \end{enumerate}
\end{proposition}


\begin{proof}
  Again, the proof is a straightforward generalisation of the proof for
  Voronoi-tilings in the quadratic case
  in~\cite{Heuberger-Krenn:2012:wnaf-analysis}.
\end{proof}


Note that $\delta = n-1$ if $U$ is, for example, a ball or a polyhedron.


\section{The Characteristic Sets}
\label{sec:sets-w_eta}


In this section we define characteristic sets for a digit at a specified
position in the \wNAF{} expansion and prove some basic properties of them. Those
will be used in the proof of Theorem~\ref{thm:countdigits}.


\begin{definition}[Characteristic Sets]\label{def:w-wk-beta}
  Let $\eta\in\cD^\bullet$. For $j\in\N_0$ define 
  \begin{equation*}
    \cW_{\eta,j} \colonequals 
    \set*{\NAFvalue{\bfxi}}{\text{
        $\bfxi\in\wNAFsetellell{0}{j+w}$ with $\xi_{-w}=\eta$}}.
  \end{equation*}
  We call $\coverV[j+w]{\cW_{\eta,j}}$ the \emph{$j$th approximation of the
    characteristic set for $\eta$}, and we define
  \begin{equation*}
    W_{\eta,j} \colonequals  \fracpartLambda{\coverV[j+w]{\cW_{\eta,j}}}.
  \end{equation*}
  Further we define the \emph{characteristic set for $\eta$}
  \begin{equation*}
    \cW_\eta \colonequals  
    \set*{\NAFvalue{\bfxi}}{\text{
        $\bfxi\in \wNAFsetinf$ with $\xi_{-w}=\eta$}}
  \end{equation*}
  and 
  \begin{equation*}
    W_\eta \colonequals  \fracpartLambda{\cW_\eta}.
  \end{equation*}

  For $j\in\N_0$ we set
  \begin{equation*}
    \beta_{\eta,j} \colonequals  \lmeas{\coverV[j+w]{\cW_{\eta,j}}} - \lmeas{\cW_\eta}.
  \end{equation*}
\end{definition}


\begin{figure}
  \centering \subfloat[$\cW_{\eta,j}$ for a lattice originating from $\tau$ with
  $\tau^2-3\tau+3=0$, $w=2$ and $j=7$]{
    \includegraphics[width=0.43\linewidth]{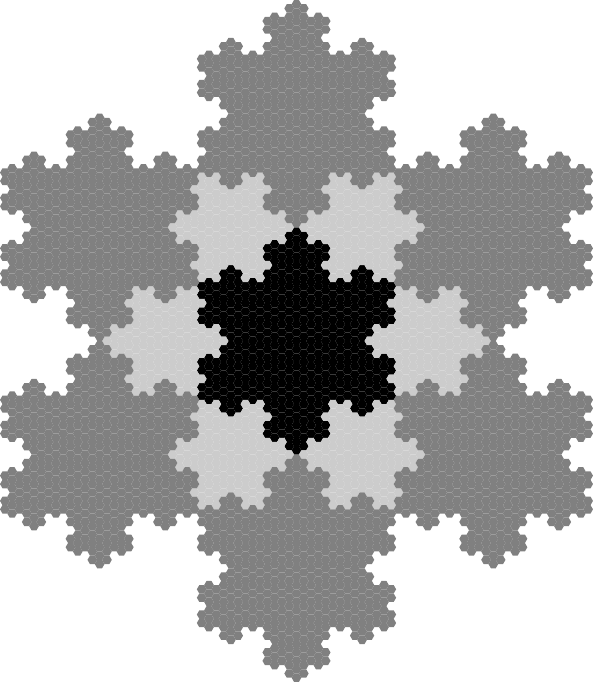}
    \label{fig:w-weta:3}}
  \hfill
  \subfloat[$\cW_{\eta,j}$ for a lattice coming from $\tau$ with
  $\tau^2-2\tau+2=0$, $w=4$ and $j=11$]{
    \includegraphics[width=0.5\linewidth]{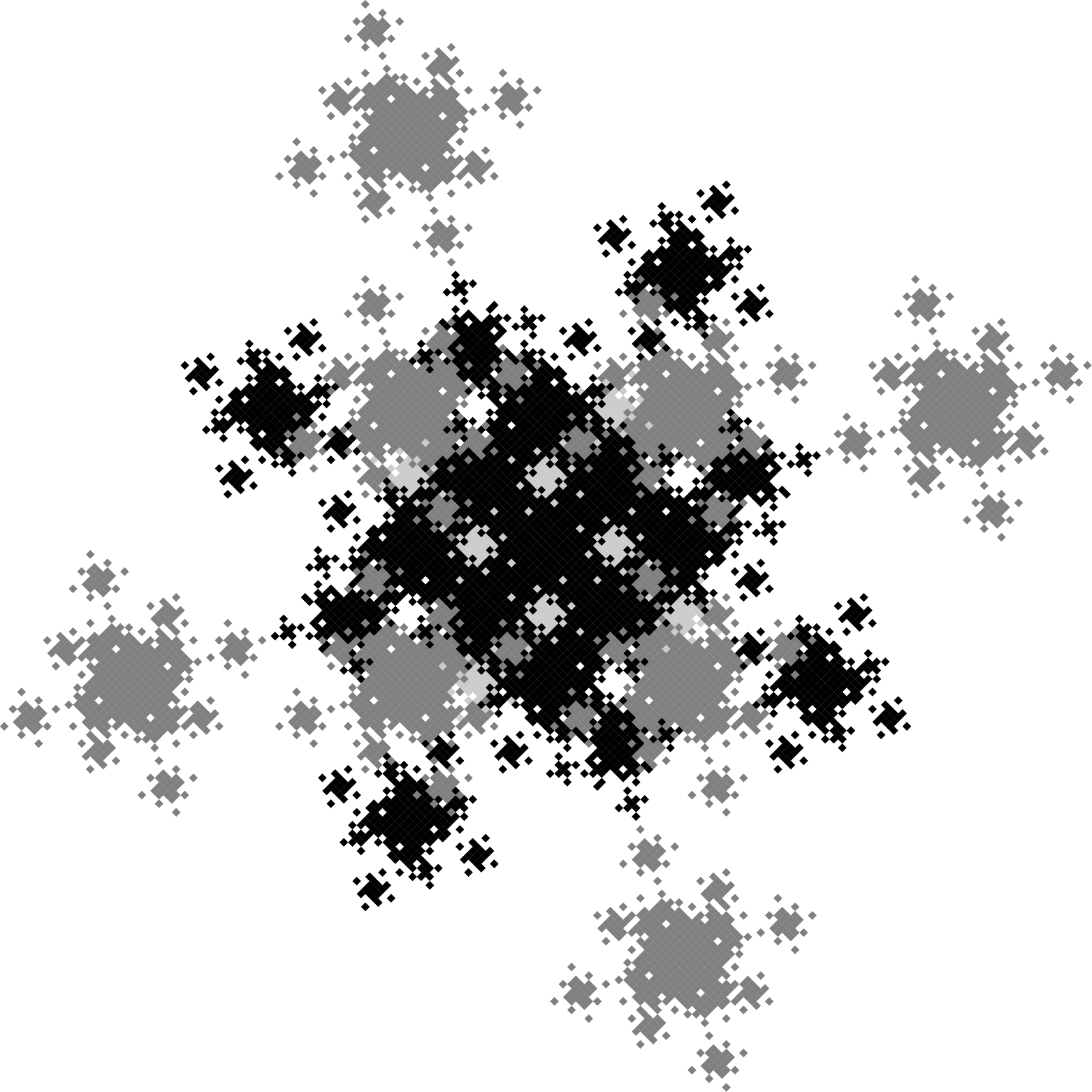}
    \label{fig:w-weta:2}}

  \caption[Characteristic sets $\cW_\eta$]{Fundamental domains and
    characteristic sets $\cW_\eta$. Each figure shows a fundamental domain. The
    light-gray coloured parts represent the approximations $\cW_{\eta,j}$ of
    the characteristic sets~$\cW_\eta$.}
  \label{fig:w-weta}
\end{figure}


Note that sometimes the set $W_\eta$ will also be called \emph{characteristic
  set for $\eta$}, and analogously for the set $W_{\eta,j}$. In
Figure~\ref{fig:w-weta} some of these characteristic sets, more precisely some
approximations of the characteristic sets, are shown.

The following proposition deals with some properties of those defined sets.


\begin{proposition}[Properties of the Characteristic Sets]
  \label{pro:prop-of-w}

  Let $\eta\in\cD^\bullet$.
  \begin{enumerate}[(a)]
  
  \item \label{enu:prop-of-w:fund-domain} We have 
    \begin{equation*}
      \cW_\eta = \eta \tau^{-w} + \Phi^{-2w+1} \cF.
    \end{equation*}
  
  \item \label{enu:prop-of-w:compact}  The set $\cW_\eta$ is compact.

  \item \label{enu:prop-of-w:w-is-union} We get
    \begin{equation*}
      \cW_\eta 
      = \closure{\bigcup_{j\in\N_0} \cW_{\eta,j}}
      = \closure{\lim_{j\to\infty} \cW_{\eta,j}}.
    \end{equation*}

  \item \label{enu:prop-of-w:lim} The set $\coverV[j+w]{\cW_{\eta,j}}$ is indeed
    an approximation of $\cW_\eta$, i.e., we have
    \begin{equation*}
      \cW_\eta 
      = \closure{\liminf_{j\in\N_0} \coverV[j+w]{\cW_{\eta,j}}}
      = \closure{\limsup_{j\in\N_0} \coverV[j+w]{\cW_{\eta,j}}}.
    \end{equation*}

  \item \label{enu:prop-of-w:int-in-liminf} We have $\interior*{\cW_\eta}
    \subseteq \liminf_{j\in\N_0} \coverV[j+w]{\cW_{\eta,j}}$.

  \item \label{enu:prop-of-w:w-in-v} We get $\cW_\eta - \Phi^{-w}\eta \subseteq
    T$, and for $j\in\N_0$ we obtain $\coverV[j+w]{\cW_{\eta,j}} - 
    \Phi^{-w}\eta \subseteq T$.

  \item \label{enu:prop-of-w:lmeas} For the Lebesgue measure of the
    characteristic set we obtain $\lmeas{\cW_\eta} = \lmeas{W_\eta}$ and for its
    approximation $\lmeas{\coverV[j+w]{\cW_{\eta,j}}} = \lmeas{W_{\eta,j}}$.

  \item \label{enu:prop-of-w:area-wj}
    Let $j\in\N_0$, then 
    \begin{equation*}
      \lmeas{\coverV[j+w]{\cW_{\eta,j}}} = d_\Lambda E + \Oh{\mu^j}
    \end{equation*}
    with $E$ and $\mu<1$ from Theorem~\vref{th:w-naf-distribution}.    
    
  \item \label{enu:prop-of-w:area-w}
    The Lebesgue measure of $W_\eta$ is
    \begin{equation*}
      \lmeas{W_\eta} = d_\Lambda E,
    \end{equation*}
    again with $E$ from Theorem~\ref{th:w-naf-distribution}.

  \item \label{enu:prop-of-w:beta}
    Let $j\in\N_0$. We get
    \begin{equation*}
      \beta_{\eta,j} = \int_{x \in T} \f{\left( 
          \indicator*{W_{\eta,j}} - \indicator*{W_\eta} \right)}{x} \dd x
      = \Oh{\mu^j}.
    \end{equation*}
    Again $\mu<1$ can be found in Theorem~\ref{th:w-naf-distribution}.

  \end{enumerate}
\end{proposition}


\begin{proof}
  The proof is a straightforward generalisation of the proof in~\cite{Heuberger-Krenn:2012:wnaf-analysis}.
\end{proof}


We can also determine the Lebesgue measure of the fundamental domain $\cF$
defined in Section~\ref{sec:fundamental-domain}.


\begin{remark}[Lebesgue Measure of the Fundamental Domain]
  \label{rem:meas-fund-domain}

  We get
  \begin{equation*}
    \lmeas{\cF} = \rho^{n(2w-1)} E d_\Lambda
    = \frac{\rho^{nw} d_\Lambda}{(\rho^n-1)w+1},
  \end{equation*}
  using~\itemref{enu:prop-of-w:fund-domain} and~\itemref{enu:prop-of-w:area-w}
  from Proposition~\ref{pro:prop-of-w} and $E$ from
  Theorem~\vref{th:w-naf-distribution}.
\end{remark}


The next lemma makes the connection between the \wNAF{}s of elements of the
lattice $\Lambda$ and the characteristic sets $W_{\eta,j}$.


\begin{lemma}\label{lem:equiv-digit-char-set}
  Let $\eta\in\cD^\bullet$, $j\geq0$. Let $z\in\Lambda$ and let
  $\bfxi\in\wNAFsetfinell{0}$ be its \wNAF{}. Then the following statements are
  equivalent:
  \begin{enumequivalences}
  \item \label{enu:e-d-cs:1} 
    The $j$th digit of $\bfxi$ equals $\eta$.
  \item \label{enu:e-d-cs:2} 
    The condition $\fracpartLambda{\Phi^{-(j+w)}z} \in W_{\eta,j}$ holds.
  \item \label{enu:e-d-cs:3}
    The inclusion $\fracpartLambda{\Phi^{-(j+w)}T_z} \subseteq W_{\eta,j}$ holds.
  \end{enumequivalences}
\end{lemma}


\begin{proof}
  The proof is a straightforward generalisation of the proof of the quadratic
  case in~\cite{Heuberger-Krenn:2012:wnaf-analysis}.
\end{proof}


\section{Counting the Occurrences of a non-zero Digit in a Region}
\label{sec:counting-digits-region}


In this section we will prove our main result on the asymptototic number of
occurrences of a digit in a given region.

Note that Iverson's notation
$\iverson{\ifdissertationkrenn\mathrm{expr}\else\var{expr}\fi}=1$ if
$\ifdissertationkrenn\mathrm{expr}\else\var{expr}\fi$ is true and
$\iverson{\ifdissertationkrenn\mathrm{expr}\else\var{expr}\fi}=0$ otherwise,
cf.\ Graham, Knuth and Patashnik~\cite{Graham-Knuth-Patashnik:1994}, will be
used.


\begin{theorem}[Counting Theorem]\label{thm:countdigits}
  Let $0 \neq \eta \in \cD$ and $N\in\R$ with $N>0$. Further let $U \subseteq
  \R^n$ be measurable with respect to the Lebesgue measure and bounded with $U
  \subseteq \ball{0}{d}$ for a finite~$d$, and set $\delta$ such that
  $\cardV{\boundaryV{NU}} = \Oh{N^\delta}$ with $1\leq\delta<n$. We denote the
  number of occurrences of the digit $\eta$ in all integer width\nbd-$w$
  non-adjacent forms with value in the region~$NU$ by
  \begin{equation*}
    \f{Z_\eta}{N}
    = \sum_{z \in NU\cap\Lambda} \sum_{j\in\N_0} 
    \iverson*{$j$th digit of $z$ in its \wNAF{}-expansion equals $\eta$}.
  \end{equation*}
  Then we get  
  \begin{equation*}
      \f{Z_\eta}{N} 
      = N^n \lmeas{U} E \log_\rho N
      + N^n \f{\psi_\eta}{\log_\rho N}
      + \Oh{N^{\alpha} \log_\rho N} 
      + \Oh{N^\delta \log_\rho N},
  \end{equation*}
  in which the expressions described below are used. The Lebesgue measure
  on~$\R^n$ is denoted by $\lambda$. We have the constant of the expectation
  \begin{equation*}
    E = \frac{1}{\rho^{n(w-1)}((\rho^n-1)w+1)},    
  \end{equation*}
  cf.\ Theorem~\vref{th:w-naf-distribution}. Then there is the function
  \begin{equation*}
    \f{\psi_\eta}{x} = \f{\psi_{\eta,\cM}}{x} 
    + \f{\psi_{\eta,\cP}}{x} + \f{\psi_{\eta,\cQ}}{x},
  \end{equation*}
  where
  \begin{equation*}  
    \f{\psi_{\eta,\cM}}{x} 
    = \lmeas{U} \left(J_0 + 1 - \fracpart{x} \right) E,
  \end{equation*}
  \begin{equation*}
    \f{\psi_{\eta,\cP}}{x} =
    \frac{\rho^{n(J_0-\fracpart{x})}}{d_\Lambda} \sum_{j=0}^\infty 
    \int_{y \in \fracpartV[j-w]{\Phi^{-\floor{x}-J_0}\rho^x U}}
    \left( \indicator{W}{\fracpartLambda{\Phi^{j-w}y}} - \lmeas{W} \right) \dd y,
  \end{equation*}
  and
  \begin{equation*}
    \psi_{\eta,\cQ} = \frac{\lmeas{U}}{d_\Lambda^2}
    \sum_{j=0}^\infty \beta_j.
  \end{equation*}
  We have $\alpha = n + \log_\rho \mu < n$, with $\mu = \left( 1 +
    \frac{1}{\rho^n w^3} \right)^{-1} < 1$, and
  \begin{equation*}
    J_0 =  \floor{ \log_\rho d - \log_\rho B_L } + 1 
  \end{equation*}
  with the constant $B_L$ of Proposition~\vref{pro:lower-bound-fracnafs}.

  Further, let
  \begin{equation*}
    \Phi = Q \diag{\rho e^{i\theta_1}, \dots, \rho e^{i\theta_n}} Q^{-1},
  \end{equation*}
  where $Q$ is a regular matrix. If there is a $p\in\N$ such that
  \begin{equation*}
    Q \diag{e^{i\theta_1 p}, \dots, e^{i\theta_n p}} Q^{-1} U = U,
  \end{equation*}
  then $\psi_\eta$ is \periodic{p}. Moreover, if $\psi_\eta$ is \periodic{p}
  for some $p\in\N$, then it is also continuous.
\end{theorem}


\begin{remark}\label{rem:main-term}
  Consider the main term of our result. When $N$ tends to infinity, we get the
  asymptotic formula
  \begin{equation*}
    Z_\eta \asymptotic N^n \lmeas{U} E \log_\rho N.
  \end{equation*}
  This result is not surprising, since intuitively the number of lattice points
  in the region $NU$ corresponds to the Lebesgue measure $N^n \lmeas{U}$ of
  this region, and each of those elements can be represented as an integer
  \wNAF{} with length about $\log_\rho N$. Therefore, using the expectation of
  Theorem~\vref{th:w-naf-distribution}, we get an explanation for this term.
\end{remark}


\begin{remark}\label{rem:delta2}
  If $\delta=n$ in the theorem, then the statement stays true, but degenerates
  to
  \begin{equation*}
    \f{Z_\eta}{N} = \Oh{N^n \log_{\abs\tau} N}.
  \end{equation*}
  This is a trivial result of Remark~\ref{rem:main-term}.
\end{remark}


The proof of Theorem~\ref{thm:countdigits} follows the ideas used by
Delange~\cite{Delange:1975:chiffres}. By Remark~\ref{rem:delta2} we restrict
ourselves to the case $\delta<n$.

We will use the following abbreviations. We omit the index $\eta$, i.e., we set
$\f{Z}{N} \colonequals \f{Z_\eta}{N}$, $W\colonequals W_\eta$ and
$W_j\colonequals W_{\eta,j}$, and further we set $\beta_j\colonequals
\beta_{\eta,j}$, cf.\ Proposition~\ref{pro:prop-of-w}.  By $\log$ we will
denote the logarithm to the base $\rho$, i.e., $\log x=\log_\rho x$. These
abbreviations will be used throughout the remaining section.


\begin{proof}[Proof of Theorem~\ref{thm:countdigits}]
  By assumption every element of $\Lambda$ is represented by a unique element
  of $\wNAFsetfinell{0}$. To count the occurrences of the digit $\eta$ in $NU$,
  we sum up~$1$ over all lattice points $z \in NU \cap \Lambda$ and for each
  $z$ over all digits in the corresponding \wNAF{} equal to $\eta$. Thus we get
  \begin{equation*}
    \f{Z}{N} 
    = \sum_{z \in NU \cap \Lambda}
    \sum_{j\in\N_0} \iverson*{$j$th digit of \wNAF{} of $z$ equals $\eta$}.
  \end{equation*}
  The inner sum over $j\in\N_0$ is finite, we will choose a large enough upper
  bound $J$ later in Lemma~\ref{lem:choosing-j}.

  Using
  \begin{equation*}
    \iverson*{$j$th digit of \wNAF{} of $z$ equals $\eta$}
    = \indicator{W_j}{\fracpartLambda{\Phi^{-j-w}z}}
  \end{equation*}
  from Lemma~\ref{lem:equiv-digit-char-set} yields
  \begin{equation*}
      \f{Z}{N} = \sum_{j=0}^J \sum_{z\in NU \cap \Lambda} 
      \indicator{W_j}{\fracpartLambda{\Phi^{-j-w}z}},
  \end{equation*}
  where additionally the order of summation was changed. This enables us to
  rewrite the sum over $z$ as an integral
  \begin{equation*}
    \begin{split}
      \f{Z}{N} &= \sum_{j=0}^J \sum_{z\in NU \cap \Lambda} 
      \frac{1}{\lmeas{T_z}} \int_{x\in T_z}
      \indicator{W_j}{\fracpartLambda{\Phi^{-j-w}x}} \dd x \\
      &= \frac{1}{\lmeas{T}} \sum_{j=0}^J \int_{x\in \coverV{NU}} 
      \indicator{W_j}{\fracpartLambda{\Phi^{-j-w}x}} \dd x.
    \end{split}
  \end{equation*}
  We split up the integrals into the ones over $NU$ and others over the
  remaining region and get
  \begin{equation*}
    \f{Z}{N} = \frac{1}{\lmeas{T}} \sum_{j=0}^J \int_{x \in NU} 
    \indicator{W_j}{\fracpartLambda{\Phi^{-j-w}x}} \dd x 
    + \f{\cF_\eta}{N},
  \end{equation*}
  in which $\f{\cF_\eta}{N}$ contains all integrals (with appropriate signs)
  over regions $\coverV{NU} \setminus NU$ and $NU \setminus \coverV{NU}$.

  By substituting $x = \Phi^J y$, $\dd x = \abs{\det\Phi}^J\dd y = \rho^{nJ}\dd
  y$ we obtain
  \begin{equation*}
    \f{Z}{N} = \frac{\rho^{nJ}}{\lmeas{T}} \sum_{j=0}^J 
    \int_{y \in \Phi^{-J}NU} \indicator{W_j}{\fracpartLambda{\Phi^{J-j-w}y}} \dd y  
    + \f{\cF_\eta}{N}.
  \end{equation*}
  Reversing the order of summation yields
  \begin{equation*}
    \f{Z}{N} = \frac{\rho^{nJ}}{\lmeas{T}} \sum_{j=0}^J 
    \int_{y \in \Phi^{-J}NU} 
    \indicator{W_{J-j}}{\fracpartLambda{\Phi^{j-w}y}} \dd y 
    + \f{\cF_\eta}{N}.
  \end{equation*}
  We rewrite this as
  \begin{equation*}
    \begin{split}
      \f{Z}{N} 
      &= \frac{\rho^{nJ}}{\lmeas{T}} (J+1) 
      \lmeas{W} \int_{y \in \Phi^{-J}NU} \dd y \\
      &\phantom{=}+ \frac{\rho^{nJ}}{\lmeas{T}} \sum_{j=0}^J \int_{y \in \Phi^{-J}NU} 
      \left( \indicator{W}{\fracpartLambda{\Phi^{j-w}y}} 
        - \lmeas{W} \right) \dd y \\
      &\phantom{=}+ \frac{\rho^{nJ}}{\lmeas{T}} \sum_{j=0}^J \int_{y \in \Phi^{-J}NU} 
      \left( \indicator{W_{J-j}}{\fracpartLambda{\Phi^{j-w}y}}
        - \indicator{W}{\fracpartLambda{\Phi^{j-w}y}} \right) \dd y \\
      &\phantom{=}+ \f{\cF_\eta}{N}.    
    \end{split}
  \end{equation*}
  With $\Phi^{-J}NU = \floorV[j-w]{\Phi^{-J}NU} \cup
  \fracpartV[j-w]{\Phi^{-J}NU}$ for each area of integration we get
  \begin{equation*}
    \f{Z}{N} 
    = \f{\cM_\eta}{N} 
    + \f{\cZ_\eta}{N} 
    + \f{\cP_\eta}{N} 
    + \f{\cQ_\eta}{N} 
    + \f{\cS_\eta}{N}
    + \f{\cF_\eta}{N},
  \end{equation*}
  in which $\cM_\eta$ is \emph{``The Main Part''}, see
  Lemma~\ref{lem:the-main-part},
  \begin{subequations}
    \label{eq:mainth:parts}
    \begin{equation}
      \label{eq:mainth:main-part}
      \f{\cM_\eta}{N} 
      = \frac{\rho^{nJ}}{\lmeas{T}} (J+1) \lmeas{W} 
      \int_{y \in \Phi^{-J}NU} \dd y,
    \end{equation}
    $\cZ_\eta$ is \emph{``The Zero Part''}, see Lemma~\ref{lem:the-zero-part},
    \begin{equation}
      \label{eq:mainth:zero-part}
      \f{\cZ_\eta}{N} 
      = \frac{\rho^{nJ}}{\lmeas{T}} \sum_{j=0}^J 
      \int_{y \in \floorV[j-w]{\Phi^{-J}NU}}
      \left( \indicator{W}{\fracpartLambda{\Phi^{j-w}y}} 
        - \lmeas{W} \right) \dd y,
    \end{equation}
    $\cP_\eta$ is \emph{``The Periodic Part''}, see
    Lemma~\ref{lem:the-periodic-part},
    \begin{equation}
      \label{eq:mainth:periodic-part}
      \f{\cP_\eta}{N} 
      = \frac{\rho^{nJ}}{\lmeas{T}} \sum_{j=0}^J 
      \int_{y \in \fracpartV[j-w]{\Phi^{-J}NU}}
      \left( \indicator{W}{\fracpartLambda{\Phi^{j-w}y}} 
        - \lmeas{W} \right) \dd y,
    \end{equation}
    $\cQ_\eta$ is \emph{``The Other Part''}, see
    Lemma~\ref{lem:the-other-part},
    \begin{equation}
      \label{eq:mainth:other-part}
      \f{\cQ_\eta}{N} 
      = \frac{\rho^{nJ}}{\lmeas{T}} \sum_{j=0}^J 
      \int_{y \in \floorV[j-w]{\Phi^{-J}NU}} 
      \f{\left( \indicator*{W_{J-j}} - \indicator*{W} 
        \right)}{\fracpartLambda{\Phi^{j-w}y}} \dd y,
    \end{equation}
    $\cS_\eta$ is \emph{``The Small Part''}, see
    Lemma~\ref{lem:the-small-part},
    \begin{equation}
      \label{eq:mainth:small-part}
      \f{\cS_\eta}{N}
      = \frac{\rho^{nJ}}{\lmeas{T}} \sum_{j=0}^J 
      \int_{y \in \fracpartV[j-w]{\Phi^{-J}NU}} 
      \f{\left( \indicator*{W_{J-j}} - \indicator*{W} 
        \right)}{\fracpartLambda{\Phi^{j-w}y}} \dd y
    \end{equation}
    and $\cF_\eta$ is \emph{``The Fractional Cells Part''}, see
    Lemma~\ref{lem:the-fractionalcells-part},
    \begin{equation}
      \label{eq:mainth:fraccells-part}
      \begin{split}
        \f{\cF_\eta}{N}
        &= \frac{1}{\lmeas{T}} \sum_{j=0}^J
        \int_{x \in \coverV{NU} \setminus NU} 
        \indicator{W_j}{\fracpartLambda{\Phi^{-j-w}x}} \dd x \\
        &\phantom{=}- \frac{1}{\lmeas{T}} \sum_{j=0}^J
        \int_{x \in NU \setminus \coverV{NU}} 
        \indicator{W_j}{\fracpartLambda{\Phi^{-j-w}x}} \dd x.
      \end{split}
    \end{equation}
  \end{subequations}
  To complete the proof we have to deal with the choice of $J$, see
  Lemma~\ref{lem:choosing-j}, as well as with each of the parts in
  \eqref{eq:mainth:parts}, see
  Lemmata~\ref{lem:the-main-part} to~\ref{lem:the-fractionalcells-part}. The
  continuity of $\psi_\eta$ is checked in Lemma~\vref{lem:psi-continous}.
\end{proof}


\begin{lemma}[Choosing $J$]\label{lem:choosing-j}
  Let $N \in \R_{\geq0}$. Then every \wNAF{} of $\wNAFsetfinell{0}$ with value
  in $NU$ has at most $J+1$ digits, where
  \begin{equation*}
    J = \floor{\log N} + J_0
  \end{equation*}
  with
  \begin{equation*}
    J_0 = \floor{ \log d - \log B_L } + 1
  \end{equation*}
  with $B_L$ of Proposition~\vref{pro:lower-bound-fracnafs}.
\end{lemma}


\begin{proof}
  Let $z \in NU$, $z\neq0$, with its corresponding \wNAF{}
  $\bfxi\in\wNAFsetfinell{0}$, and let $j\in\N_0$ be the largest index such
  that the digit $\xi_j$ is non-zero. By using
  Corollary~\vref{cor:bounds-value}, we conclude that
  \begin{equation*}
    \rho^j B_L \leq \norm{z} < Nd.
  \end{equation*}
  This means
  \begin{equation*}
    j < \log N + \log d - \log B_L,
  \end{equation*}
  and thus we have
  \begin{equation*}
    j \leq \floor{\log N + \log d - \log B_L} 
    \leq \floor{\log N} + \floor{\log d - \log B_L} + 1.
  \end{equation*}
  Defining the right hand side of this inequality as $J$ finishes the proof.
\end{proof}


\begin{remark}\label{rem:auxcalc-taujn}
  For the parameter used in the region of integration in the proof of
  Theorem~\ref{thm:countdigits} we get
  \begin{equation*}
    \abs{\det\left(\Phi^{-J} N\right)} = \Oh{1}.
  \end{equation*} 
  In particular, we get $\norm{\Phi^{-J} N} = \Oh{1}$.
\end{remark}


\begin{proof}
  We have
  \begin{equation*}
    \abs{\det\left(\Phi^{-J}N\right)} = \left(\rho^{-J}N\right)^n.
  \end{equation*}
  With $J$ of Lemma~\ref{lem:choosing-j} we obtain
  \begin{equation*}
    \rho^{-J} N = \rho^{-\floor{\log N}-J_0} \rho^{\log N}
    = \rho^{\log N - \floor{\log N} - J_0} 
    = \rho^{\fracpart{\log N} - J_0}.
  \end{equation*}
  Since $\rho^{\fracpart{\log N} - J_0}$ is bounded by $\rho^{1-J_0}$, it is
  $\Oh{1}$. Therefore $\det\left(\Phi^{-J}N\right)$ is $\Oh{1}$. Since
  $\norm{\Phi^{-1}} = \rho^{-1}$ we conclude that $\norm{\Phi^{-J} N}$ is
  $\Oh{1}$.
\end{proof}


\begin{lemma}[The Main Part]\label{lem:the-main-part}
  For \eqref{eq:mainth:main-part} in the proof of Theorem~\ref{thm:countdigits}
  we get
  \begin{equation*}
    \f{\cM_\eta}{N}
    = N^n \lmeas{U} E \log N
    + N^n \f{\psi_{\eta,\cM}}{\log N} 
  \end{equation*}
  with a \periodic{1} function $\psi_{\eta,\cM}$,
  \begin{equation*}  
    \f{\psi_{\eta,\cM}}{x} 
    = \lmeas{U} \left( J_0 + 1 - \fracpart{x} \right) E
  \end{equation*}
  and $E$ of Theorem~\vref{th:w-naf-distribution}.
\end{lemma}


\begin{proof}
  We have
  \begin{equation*}
    \f{\cM_\eta}{N}
    = \frac{\rho^{nJ}}{\lmeas{T}} (J+1) \lmeas{W} 
    \int_{y \in \Phi^{-J}NU} \dd y.
  \end{equation*}
  As $\lmeas{\Phi^{-J}NU} = \rho^{-nJ}N^n\lmeas{U}$ we obtain
  \begin{equation*}
    \f{\cM_\eta}{N}
    = \frac{\lmeas{W}}{\lmeas{T}} (J+1) N^n \lmeas{U}.
  \end{equation*}
  By taking $\lmeas{W} = \lmeas{T} E$ from \itemref{enu:prop-of-w:area-w} of
  Proposition~\vref{pro:prop-of-w} and $J$ from Lemma~\ref{lem:choosing-j} we
  get
  \begin{equation*}
    \f{\cM_\eta}{N}
    = N^n \lmeas{U} E \left( \floor{\log N} 
        + J_0 + 1 \right).
  \end{equation*}
  Finally, the desired result follows by using $\floor{x} = x - \fracpart{x}$.
\end{proof}


\begin{lemma}[The Zero Part]\label{lem:the-zero-part}
  For \eqref{eq:mainth:zero-part} in the proof of
  Theorem~\ref{thm:countdigits} we get
  \begin{equation*}
    \f{\cZ_\eta}{N} = 0.
  \end{equation*}
\end{lemma}


\begin{proof}
  Consider the integral
  \begin{equation*}    
      I_j
      \colonequals  \int_{y \in \floorV[j-w]{\Phi^{-J}NU}}
      \left( \indicator{W}{\fracpartLambda{\Phi^{j-w}y}} 
        - \lmeas{W} \right) \dd y.
  \end{equation*}
  We can rewrite the region of integration as
  \begin{equation*}
    \floorV[j-w]{\Phi^{-J}NU} 
    = \Phi^{-(j-w)} \floorV{\Phi^{j-w} \Phi^{-J}NU}
    = \Phi^{-(j-w)} \bigcup_{z\in R_{j-w}} T_z
  \end{equation*}
  for some appropriate $R_{j-w} \subseteq \Lambda$. Substituting $x =
  \Phi^{j-w} y$, $\dd x = \rho^{n(j-w)} \dd y$ yields
  \begin{equation*}
      I_j
      = \rho^{-n(j-w)} \int_{x \in \bigcup_{z\in R_{j-w}} T_z}
      \left( \indicator{W}{\fracpartLambda{x}} 
        - \lmeas{W} \right) \dd x.    
  \end{equation*}
  We split up the integral and eliminate the fractional part
  $\fracpartLambda{x}$ by translation to get
  \begin{equation*}
      I_j
      = \rho^{-n(j-w)} \sum_{z\in R_{j-w}} 
      \underbrace{\int_{x \in T} \left( \indicator{W}{x}  
          - \lmeas{W} \right) \dd x}_{=0}.    
  \end{equation*}
  Thus, for all $j\in\N_0$ we obtain $I_j=0$, and therefore
  $\f{\cZ_\eta}{N}=0$.
\end{proof}


\begin{lemma}[The Periodic Part]\label{lem:the-periodic-part}
  For \eqref{eq:mainth:periodic-part} in the proof of
  Theorem~\ref{thm:countdigits} we get
  \begin{equation*}
    \f{\cP_\eta}{N} = 
    N^n \f{\psi_{\eta,\cP}}{\log N} + \Oh{N^\delta}
  \end{equation*}
  with a function $\psi_{\eta,\cP}$,
  \begin{equation*}
    \f{\psi_{\eta,\cP}}{x} =
    \frac{\rho^{n(J_0-\fracpart{x})}}{\lmeas{T}} \sum_{j=0}^\infty 
    \int_{y \in \fracpartV[j-w]{\Phi^{-\floor{x}-J_0}\rho^x U}}
    \left( \indicator{W}{\fracpartLambda{\Phi^{j-w}y}} - \lmeas{W} \right) \dd y.
  \end{equation*}
  Let
  \begin{equation*}
    \Phi = Q \diag{\rho e^{i\theta_1}, \dots, \rho e^{i\theta_n}} Q^{-1},
  \end{equation*}
  where $Q$ is a regular matrix. If there is a $p\in\N$ such that
  \begin{equation}\label{eq:cond-periodic}
    Q \diag{e^{i\theta_1 p}, \dots, e^{i\theta_n p}} Q^{-1} U = U,
  \end{equation}
  then $\psi_{\eta,\cP}$ is \periodic{p}.
\end{lemma}


\begin{proof}
  Consider
  \begin{equation*}
    I_j
    \colonequals  \int_{y \in \fracpartV[j-w]{\Phi^{-J}NU}}
      \left( \indicator{W}{\fracpartLambda{\Phi^{j-w}y}}
        - \lmeas{W} \right) \dd y.
  \end{equation*}
  The region of integration satisfies
  \begin{equation}
    \label{eq:proof-periodic-part:region}
    \begin{split}
      \fracpartV[j-w]{\Phi^{-J}NU} 
      \subseteq \boundaryV[j-w]{\Phi^{-J}NU} 
      = \Phi^{-(j-w)} \bigcup_{z\in R_{j-w}} T_z
    \end{split}
  \end{equation}
  for some appropriate $R_{j-w} \subseteq \Lambda$. 
  
  We use the triangle inequality and substitute $x = \Phi^{j-w} y$, $\dd x =
  \rho^{n(j-w)} \dd y$ in the integral to get
  \begin{equation*}
    \abs{I_j} \leq 
    \rho^{-n(j-w)} \int_{x \in \bigcup_{z\in R_{j-w}} T_z}
      \underbrace{\abs{ \indicator{W}{\fracpartLambda{x}} 
          - \lmeas{W}}}_{\leq 1+\lmeas{W}} \dd x.
  \end{equation*}
  After splitting up the integral and using translation to eliminate the
  fractional part, we get
  \begin{equation*}
    \abs{I_j} \leq \rho^{-n(j-w)} \left(1+\lmeas{W}\right)
    \sum_{z\in R_{j-w}} \int_{x \in T} \dd x 
    = \rho^{-n(j-w)} \left(1+\lmeas{W}\right) \lmeas{T} \card{R_{j-w}}.
  \end{equation*}
  Using $\cardV{\boundaryV{\Psi U}} = \Oh{\abs{\det\Psi}^{\delta/n}}$ as
  assumed and~\eqref{eq:proof-periodic-part:region} we gain
  \begin{equation*}
    \card{R_{j-w}} = \abs{\det\left(\Phi^{-J}N\Phi^{j-w}\right)}^{\delta/n}
    = \Oh{\rho^{(j-w)\delta}},
  \end{equation*}
  because $\abs{\det\left(\Phi^{-J}N\right)} = \Oh{1}$, see
  Remark~\ref{rem:auxcalc-taujn}, and $\abs{\det\Phi}=\rho^n$. Thus
  \begin{equation*}
    \abs{I_j} = \Oh{\rho^{\delta (j-w) - n(j-w)}} 
    = \Oh{\rho^{(\delta-n)j}}.
  \end{equation*}

  Now we want to make the summation in $\cP_\eta$ independent from
  $J$, so we consider 
  \begin{equation*}
    I \colonequals  \frac{\rho^{nJ}}{\lmeas{T}} \sum_{j=J+1}^\infty I_j
  \end{equation*}
  Again we use the triangle inequality and we calculate the sum to obtain
  \begin{equation*}
    \abs{I} 
    = \Oh{\rho^{nJ}} \sum_{j=J+1}^\infty \Oh{\rho^{(\delta-n) j}}
    = \Oh{\rho^{nJ} \rho^{(\delta-n) J}}
    = \Oh{\rho^{\delta J}}.
  \end{equation*}
  Note that $\Oh{\rho^J} = \Oh{N}$, so
  we obtain $\abs{I} = \Oh{N^\delta}$.

  Let us look at the growth of 
  \begin{equation*}
    \f{\cP_\eta}{N} 
    = \frac{\rho^{nJ}}{\lmeas{T}} \sum_{j=0}^J I_j.
  \end{equation*}
  We get
  \begin{equation*}
    \abs{\f{\cP_\eta}{N}} 
    = \Oh{\rho^{nJ}} \sum_{j=0}^J \Oh{\rho^{(\delta-n) j}}
    = \Oh{\rho^{nJ}} = \Oh{N^n},
  \end{equation*}
  using $\delta<n$.
 
  Finally, inserting $J$ from Lemma~\ref{lem:choosing-j} and extending the sum
  to infinity, as described above, yields
  \begin{equation*}
    \begin{split}
      \f{\cP_\eta}{N}
      &= \frac{\rho^{nJ}}{\lmeas{T}} \sum_{j=0}^J 
      \int_{y \in \fracpartV[j-w]{\Phi^{-J}NU}}
      \left( \indicator{W}{\fracpartLambda{\Phi^{j-w}y}} 
        - \lmeas{W} \right) \dd y \\
      &= N^n  \f{\psi_{\eta,\cP}}{\log N} + \Oh{N^\delta}.
    \end{split}
  \end{equation*}
  with the desired $\psi_{\eta,\cP}$.

  Now suppose~\eqref{eq:cond-periodic} holds. Then
  \begin{align*}
    \Phi^{-\floor{x}-J_0}\rho^x U
    &= \rho^x
    Q \diag{\rho^{-\floor{x}-J_0} e^{-i\theta_1\left(\floor{x}+J_0\right)}, \dots, 
      \rho^{-\floor{x}-J_0} e^{-i\theta_n\left(\floor{x}+J_0\right)}} Q^{-1} U \\
    &= \rho^{\fracpart{x}-J_0} 
    Q \diag{e^{-i\theta_1\left(\floor{x}+J_0\right)}, \dots, 
      e^{-i\theta_n\left(\floor{x}+J_0\right)}} Q^{-1} U.
  \end{align*}
  Now, we can conclude that the region of integration in
  $\f{\psi_{\eta,\cP}}{x}$ is \periodic{p} using~\eqref{eq:cond-periodic}. All
  other occurrences of $x$ in $\f{\psi_{\eta,\cP}}{x}$ are of the form
  $\fracpart{x}$, i.e., \periodic{1}, so period~$p$ is obtained.
\end{proof}


\begin{lemma}[The Other Part]\label{lem:the-other-part}
  For \eqref{eq:mainth:other-part} in the proof of
  Theorem~\ref{thm:countdigits} we get
  \begin{equation*}
    \f{\cQ_\eta}{N} = N^n \psi_{\eta,\cQ}
    + \Oh{N^{\alpha} \log N}  + \Oh{N^\delta},
  \end{equation*}
  with
  \begin{equation*}
    \psi_{\eta,\cQ} = \frac{\lmeas{U}}{\lmeas{T}}
    \sum_{j=0}^\infty \frac{\beta_j}{\lmeas{T}}
  \end{equation*}
  and $\alpha = n+\log\mu<n$, where $\mu<1$ can be found in
  Theorem~\vref{th:w-naf-distribution}.
\end{lemma}


\begin{proof}
  Consider
  \begin{equation*}
      I_{j,\ell}
      \colonequals  \int_{y \in \floorV[j-w]{\Phi^{-J}NU}} 
      \f{\left( \indicator*{W_{\eta,\ell}} - \indicator*{W} 
        \right)}{\fracpartLambda{\Phi^{j-w}y}} \dd y.
  \end{equation*}
  We can rewrite the region of integration and get
  \begin{equation*}
    \floorV[j-w]{\Phi^{-J}NU} 
    = \Phi^{-(j-w)} \floorV{\Phi^{j-w} \Phi^{-J}NU}
    = \Phi^{-(j-w)} \bigcup_{z\in R_{j-w}} T_z
  \end{equation*}
  for some appropriate $R_{j-w} \subseteq \Lambda$, as in the proof of
  Lemma~\ref{lem:the-zero-part}. Substituting $x = \Phi^{j-w} y$, $\dd x =
  \rho^{n(j-w)} \dd y$ yields
  \begin{equation*}
      I_{j,\ell}
      = \rho^{-n(j-w)} \int_{x \in \bigcup_{z\in R_{j-w}} T_z} 
      \f{\left( \indicator*{W_{\eta,\ell}} - \indicator*{W} 
        \right)}{\fracpartLambda{x}} \dd x
  \end{equation*}
  and further
  \begin{equation*}
      I_{j,\ell}
      = \rho^{-n(j-w)} \sum_{z\in R_{j-w}} 
      \underbrace{\int_{x \in T} \f{\left( \indicator*{W_{\eta,\ell}} 
            - \indicator*{W} \right)}{x} \dd x}_{= \beta_\ell}
      = \rho^{-n(j-w)} \card{R_{j-w}} \beta_\ell,
  \end{equation*}
  by splitting up the integral, using translation to eliminate the fractional
  part and taking $\beta_\ell$ according to \itemref{enu:prop-of-w:beta} of
  Proposition~\vref{pro:prop-of-w}. From Proposition~\vref{pro:set-nu} we
  obtain
  \begin{equation*}
    \frac{\card{R_{j-w}}}{\rho^{n(j-w)}}
    = \frac{\abs{\det\left(\Phi^{-J} N \Phi^{j-w}\right)}}{\rho^{n(j-w)}} 
    \frac{\lmeas{U}}{\lmeas{T}}
    + \Oh{\frac{\abs{\det\left(\Phi^{-J} N \Phi^{j-w}\right)}^{\delta/n}}{
        \rho^{n(j-w)}}},
  \end{equation*}
  which can be rewritten as
  \begin{equation*}
    \frac{\card{R_{j-w}}}{\rho^{n(j-w)}}
    = \rho^{-nJ} N^n \frac{\lmeas{U}}{\lmeas{T}}
    + \Oh{\rho^{(\delta-n) j}}
  \end{equation*}
  because $\abs{\det\Phi}=\rho^n$ and because $\abs{\tau^{-J}N} = \Oh{1}$, see
  Remark~\ref{rem:auxcalc-taujn}.
  
  Now let us have a look at
  \begin{equation*}
      \f{\cQ_\eta}{N} 
      = \frac{\rho^{nJ}}{\lmeas{T}} \sum_{j=0}^J I_{j,J-j}.
  \end{equation*}
  Inserting the result above and using $\beta_{\ell} = \Oh{\mu^\ell}$, see
  \itemref{enu:prop-of-w:beta} of Proposition~\vref{pro:prop-of-w}, yields
  \begin{equation*}
    \f{\cQ_\eta}{N} 
    = N^n \frac{\lmeas{U}}{\left(\lmeas{T}\right)^2}
    \sum_{j=0}^J \beta_{J-j}
    + \rho^{nJ} \sum_{j=0}^J 
    \Oh{\rho^{(\delta-n)j}} \Oh{\mu^{J-j}}.
  \end{equation*}
  Therefore, after reversing the order of the first summation, we obtain
  \begin{equation*}
    \f{\cQ_\eta}{N} 
    = N^n \frac{\lmeas{U}}{\left(\lmeas{T}\right)^2}
    \sum_{j=0}^J \beta_j
    + \rho^{nJ} \mu^J \sum_{j=0}^J 
    \Oh{\left(\mu\rho^{n-\delta}\right)^{-j}}.
  \end{equation*}
  If $\mu\rho^{n-\delta}\geq1$, then the second sum is $J \Oh{1}$,
  otherwise the sum is $\Oh{\mu^{-J} \rho^{(\delta-2)J}}$. So we obtain
  \begin{equation*}
    \f{\cQ_\eta}{N} 
    = N^n \frac{\lmeas{U}}{\left(\lmeas{T}\right)^2}
    \sum_{j=0}^J \beta_j
    + \Oh{\rho^{nJ} \mu^J J}
    + \Oh{\rho^{\delta J}}.
  \end{equation*}
  Using $J = \f{\Theta}{\log N}$, see Lemma~\ref{lem:choosing-j}, and defining
  $\alpha = n + \log\mu$ yields
  \begin{equation*}
    \f{\cQ_\eta}{N} 
    = N^n \frac{\lmeas{U}}{\left(\lmeas{T}\right)^2}
    \sum_{j=0}^J \beta_j
    + \underbrace{\Oh{N^{n + \log\mu} \log N}}_{=\Oh{N^\alpha \log N}}
    + \Oh{N^\delta}.
  \end{equation*}

  Now consider the first sum. Since $\beta_j = \Oh{\mu^j}$, see
  \itemref{enu:prop-of-w:beta} of Proposition~\vref{pro:prop-of-w}, we obtain
  \begin{equation*}
    N^n \sum_{j=J+1}^\infty \beta_j = N^n \Oh{\mu^J} 
    = \Oh{N^\alpha}.
  \end{equation*}
  Thus the lemma is proved, because we can extend the sum to infinity.
\end{proof}


\begin{lemma}[The Small Part]\label{lem:the-small-part}
  For \eqref{eq:mainth:small-part} in the proof of
  Theorem~\ref{thm:countdigits} we get
  \begin{equation*}
    \f{\cS_\eta}{N} = \Oh{N^{\alpha} \log N} + \Oh{ N^\delta }
  \end{equation*}  
  with $\alpha = n + \log\mu<n$ and $\mu<1$ from
  Theorem~\vref{th:w-naf-distribution}.
\end{lemma}


\begin{proof}
  Consider
  \begin{equation*}
    I_{j,\ell}
    \colonequals  \int_{y \in \fracpartV[j-w]{\Phi^{-J}NU}} 
    \f{\left( \indicator*{W_\ell} - \indicator*{W} 
      \right)}{\fracpartLambda{\Phi^{j-w}y}} \dd y.
    \end{equation*}
    Again, as in the proof of Lemma~\ref{lem:the-periodic-part}, the region of
    integration satisfies
  \begin{equation}
    \label{eq:proof-small-part:region}
    \fracpartV[j-w]{\Phi^{-J}NU} 
    \subseteq \boundaryV[j-w]{\Phi^{-J}NU}
    = \Phi^{-(j-w)} \bigcup_{z\in R_{j-w}} T_z
  \end{equation}
  for some appropriate $R_{j-w} \subseteq \Lambda$. 

  We substitute $x = \Phi^{j-w} y$, $\dd x = \rho^{n(j-w)} \dd y$ in the
  integral to get
  \begin{equation*}
    \abs{I_{j,\ell}} =
    \rho^{-n(j-w)} \abs{\int_{x \in \bigcup_{z\in R_{j-w}} T_z}
    \f{\left(\indicator*{W_\ell} 
      - \indicator*{W}\right)}{\fracpartLambda{x}} \dd x}.
  \end{equation*}
  Again, after splitting up the integral, using translation to eliminate the
  fractional part and the triangle inequality, we get
  \begin{equation*}
    \abs{I_{j,\ell}} \leq  \rho^{-n(j-w)}
    \sum_{z\in R_{j-w}} \underbrace{\abs{\int_{x \in T}
      \f{\left(\indicator*{W_\ell} - \indicator*{W}\right)}{x} 
      \dd x}}_{= \abs{\beta_\ell}}
    = \rho^{-n(j-w)} \card{R_{j-w}} \abs{\beta_\ell} ,
  \end{equation*}
  in which $\abs{\beta_\ell} = \Oh{\mu^\ell}$ is known from
  \itemref{enu:prop-of-w:beta} of Proposition~\vref{pro:prop-of-w}. Using
  $\cardV{\boundaryV{\Psi U}} = \Oh{\abs{\det\Psi}^{\delta/n}}$,
  Remark~\ref{rem:auxcalc-taujn}, and~\eqref{eq:proof-small-part:region} we
  get
  \begin{equation*}
    \card{R_{j-w}} = \Oh{\abs{\det\Phi^{-J}N\Phi^{j-w}}^{\delta/n}} 
    = \Oh{\rho^{\delta (j-w)}},
  \end{equation*}
  because $\abs{\det\Phi}=\rho^n$ and $\abs{\tau^{-J}N} = \Oh{1}$. Thus
  \begin{equation*}
    \abs{I_{j,\ell}} 
    = \Oh{ \mu^\ell \rho^{(\delta - n)(j-w)} } 
    = \Oh{ \mu^\ell \rho^{(\delta-n) j} } 
  \end{equation*}
  follows by assembling everything together.

  Now we are ready to analyse
  \begin{equation*}
    \f{\cS_\eta}{N}
    = \frac{\rho^{nJ}}{\lmeas{T}} \sum_{j=0}^J I_{j,J-j}.    
  \end{equation*}
  Inserting the result above yields
  \begin{equation*}
    \abs{\f{\cS_\eta}{N}}
    = \frac{\rho^{nJ}}{\lmeas{T}} \sum_{j=0}^J 
    \Oh{ \mu^{J-j} \rho^{(\delta-n) j} } 
    = \frac{\mu^J \rho^{nJ}}{\lmeas{T}} \sum_{j=0}^J 
    \Oh{ \left( \mu \rho^{n-\delta}\right)^{-j} } 
  \end{equation*}
  and thus, by the same argument as in the proof of
  Lemma~\ref{lem:the-other-part},
  \begin{equation*}
    \abs{\f{\cS_\eta}{N}}
    = \mu^J \rho^{nJ} \Oh{J + \mu^{-J} \rho^{(\delta - n)J}} 
    = \Oh{\mu^J \rho^{nJ} J} + \Oh{ \rho^{\delta J} }.
  \end{equation*}
  Finally, by using Lemma~\ref{lem:choosing-j} we obtain
  \begin{equation*}
    \abs{\f{\cS_\eta}{N}} = \Oh{N^{\alpha} \log N} + \Oh{ N^\delta }
  \end{equation*}
  with $\alpha = n+\log\mu$. Since $\mu<1$, we have $\alpha<n$.
\end{proof}


\begin{lemma}[The Fractional Cells Part]\label{lem:the-fractionalcells-part}
  For \eqref{eq:mainth:fraccells-part} in the proof of
  Theorem~\ref{thm:countdigits} we get
  \begin{equation*}
    \f{\cF_\eta}{N} = \Oh{N^\delta \log N}.
  \end{equation*}    
\end{lemma}


\begin{proof}
  For the regions of integration in $\cF_\eta$ we obtain
  \begin{subequations}
    \begin{align*}
      NU \setminus \coverV{NU}
      &\subseteq \ceilV{NU} \setminus \floorV{NU}
      = \boundaryV{NU}
      = \bigcup_{z\in R} T_z
      \intertext{and}
      \coverV{NU} \setminus NU
      &\subseteq \ceilV{NU} \setminus \floorV{NU}
      = \boundaryV{NU}
      = \bigcup_{z\in R} T_z  
    \end{align*}
  \end{subequations}
  for some appropriate $R \subseteq \Lambda$ using
  Proposition~\vref{pro:round-v-basic-prop}. Thus we get
  \begin{equation*}
    \abs{\f{\cF_\eta}{N}} 
    \leq \frac{2}{\lmeas{T}} \sum_{j=0}^J
    \int_{x \in \bigcup_{z\in R} T_z} 
    \indicator{W_j}{\fracpartLambda{\Phi^{-j-w}x}} \dd x
    \leq \frac{2}{\lmeas{T}} \sum_{j=0}^J \sum_{z\in R}
    \int_{x \in T_z} \dd x,
  \end{equation*}
  in which the indicator function was replaced by $1$. Dealing with the sums and
  the integral, which is $\Oh{1}$, we obtain
  \begin{equation*}
    \abs{\f{\cF_\eta}{N}} = (J+1) \card*{R} \Oh{1}.
  \end{equation*}
  Since $J = \Oh{\log N}$, see Lemma~\ref{lem:choosing-j}, and
  $\card*{R} = \Oh{N^\delta}$, the desired result
  follows.
\end{proof}


\begin{lemma}\label{lem:psi-continous}
  If the $\psi_\eta$ from Theorem~\ref{thm:countdigits} is \periodic{p} for
  some $p\in\N$, then $\psi_\eta$ is also continuous.
\end{lemma}


\begin{proof}
  There are two possible parts of $\psi_\eta$ where a discontinuity could
  occur: the first is $\fracpart{x}$ for an $x\in\Z$, the second is building
  $\fracpartV[j-w]{\dots}$ in the region of integration in $\psi_{\eta,\cP}$.

  The latter is no problem, i.e., no discontinuity, since
  \begin{multline*}
    \int_{y \in \fracpartV[j-w]{\Phi^{-\floor{x}-J_0}\rho^x U}}
    \left( \indicator{W}{\fracpartLambda{\Phi^{j-w}y}} - \lmeas{W} \right) 
    \dd y \\
    = \int_{y \in \Phi^{-\floor{x}-J_0}\rho^x U}
    \left( \indicator{W}{\fracpartLambda{\Phi^{j-w}y}} - \lmeas{W} \right) 
    \dd y,
  \end{multline*}
  because the integral over the region
  $\floorV[j-w]{\Phi^{-\floor{x}-J_0}\rho^x U}$ is zero, see proof of
  Lemma~\ref{lem:the-zero-part}.

  Now we deal with the continuity at $x\in\Z$. Let $m \in x + p\Z$, let
  $M=\rho^m$, and consider
  \begin{equation*}
    \f{Z_\eta}{M} - \f{Z_\eta}{M-1}.
  \end{equation*}
  For an appropriate $a\in\R$ we get
  \begin{equation*}
    \f{Z_\eta}{M} 
    = a M^n \log M + M^n \f{\psi_\eta}{\log M}
    + \Oh{M^{\alpha} \log M}
    + \Oh{M^{\delta} \log M},
  \end{equation*}
  and thus
  \begin{equation*}
    \f{Z_\eta}{M} = a M^n m 
    + M^n \underbrace{\f{\psi_\eta}{m}}_{= \f{\psi_\eta}{x}}
    + \Oh{M^\alpha m}
    + \Oh{M^\delta m}.
  \end{equation*}
  Further we obtain
  \begin{multline*}
    \f{Z_\eta}{M-1} 
    = a \left(M-1\right)^n \f{\log}{M-1}
    + \left(M-1\right)^n
    \f{\psi_\eta}{\f{\log}{M-1}} \\
    + \Oh{\left(M-1\right)^{\alpha} \f{\log}{M-1}}
    + \Oh{\left(M-1\right)^{\delta} \f{\log}{M-1}},
  \end{multline*}
  and thus, using the abbreviation $L = \f{\log}{1 - M^{-1}}$ and $\delta\geq1$,
  \begin{equation*}
    \f{Z_\eta}{M-1}
    = a M^n m 
    + M^n \underbrace{\f{\psi_\eta}{m + L}}_{= \f{\psi_\eta}{x+L}}
    + \Oh{M^\alpha m}
    + \Oh{M^\delta m}.
  \end{equation*}
  Therefore we obtain
  \begin{equation*}
    \frac{\f{Z_\eta}{M} - \f{Z_\eta}{M-1}}{M^n} 
    = \f{\psi_\eta}{x} - \f{\psi_\eta}{x + L}
    + \Oh{M^{\alpha-n} m}
    + \Oh{M^{\delta-n} m}.
  \end{equation*}
  Since $\cardV{M U \setminus \left(M-1\right) U}$ is clearly an upper bound
  for the number of \wNAF{}s with values in $M U \setminus \left(M-1\right) U$
  and each of these \wNAF{}s has at most $\floor{\log M} + J_0 + 1$ digits, see
  Lemma~\ref{lem:choosing-j}, we obtain
  \begin{equation*}
    \f{Z_\eta}{M} - \f{Z_\eta}{M-1}
    \leq \cardV{M U \setminus \left(M-1\right) U}
    \left(m + J_0 + 2\right).
  \end{equation*}
  Using \itemref{enu:set-nu:difference} of Proposition~\vref{pro:set-nu} yields
  \begin{equation*}
    \f{Z_\eta}{M} - \f{Z_\eta}{M-1} = \Oh{M^\delta m}.
  \end{equation*}
  Therefore we get
  \begin{equation*}
    \f{\psi_\eta}{x} - \f{\psi_\eta}{x + L} 
    = \Oh{M^{\delta-n} m}
    + \Oh{M^{\alpha-n} m}
    + \Oh{M^{\delta-n} m}.
  \end{equation*}
  Taking the limit $m\to\infty$ in steps of $p$, and
  using $\alpha<n$ and $\delta<n$ yields
  \begin{equation*}
    \f{\psi_\eta}{x} - \lim_{\eps \to 0^-} \f{\psi_\eta}{x+\eps} = 0, 
  \end{equation*}
  i.e., $\psi_\eta$ is continuous at $x\in\Z$.
\end{proof}


\section{Counting Digits in Conjunction with Hyperelliptic Curve Cryptography}
\label{sec:thm-hyperell}


As mentioned in the introduction, we are interested in numeral systems coming
from hyperelliptic curve cryptography. There the base is an algebraic integer,
where all conjugates have the same absolute value.

Let $H$ be a hyperelliptic curve (or more generally an algebraic curve) of
genus~$g$ defined over $\mathbb{F}_q$ (a field with $q$ elements). The
Frobenius endomorphism operates on the Jacobian variety of $H$ and satisfies a
characteristic polynomial $f\in\Z[T]$ of degree~$2g$. This polynomial fulfils
the equation
\begin{equation*}
  f(T) = T^{2g} L(1/T),
\end{equation*}
where $L(T)$ denotes the numerator of the zeta-function of $H$ over
$\mathbb{F}_q$, cf.\ Weil~\cite{Weil:1948:var-ab-et-courbes-alg,
  Weil:1971:courbes-alg-et-var-ab}. The Riemann Hypothesis of the Weil
Conjectures, cf.\ Weil~\cite{Weil:1949}, Dwork~\cite{Dwork:1960} and
Deligne~\cite{Deligne:1974}, states that all zeros of $L$ have absolute value
$1/\sqrt{q}$. Therefore all roots of $f$ have absolute value $\sqrt{q}$.

Later we suppose that $\tau$ is a root of $f$, and we consider numeral systems
with a base~$\tau$. But before, we describe getting from that setting to a
lattice, which we need in Section~\ref{sec:set-up}. This is generally known and
was also used in Heuberger and Krenn~\cite{Heuberger-Krenn:2013:general-wnaf}.


First consider a number field $K$ of degree~$n$. Denote the real
embeddings of $K$ by $\sigma_1$, \ldots, $\sigma_s$ and the non-real
complex embeddings of $K$ by $\sigma_{s+1}$,
$\overline{\sigma_{s+1}}$, \ldots, $\sigma_{s+t}$,
$\overline{\sigma_{s+t}}$, where $\overline{\,\cdot\,}$ denotes
complex conjugation and $n=s+2t$. The \emph{Minkowski map}
$\Sigma\colon K \to \R^n$ maps $\alpha\in K$ to
\begin{equation*}
  \left(\sigma_1(\alpha),\ldots,\sigma_{s}(\alpha), \Re \sigma_{s+1}(\alpha),
    \Im \sigma_{s+1}(\alpha), \ldots, \Re \sigma_{s+t}(\alpha),
    \Im \sigma_{s+t}(\alpha)\right)\in\R^n.
\end{equation*}
Now let $\tau$ be an algebraic integer of degree~$n$ (as above, where $\tau$
was supposed to be a root of the characteristic polynomial~$f$ of the Frobenius
endomorphism) and such that all its conjugates have the same absolute
value~$\rho>1$. Note that the absolute value of the field norm of~$\tau$
equals~$\rho^n$. Set $K=\f{\Q}{\tau}$ and consider the order $\Z[\tau]$. We get
a lattice $\Lambda=\f{\Sigma}{\Z[\tau]}$ of degree~$n$ in the
space~$\R^n$. Application of the map $\Phi\colon\Lambda\to\Lambda$ on a lattice
element should correspond to the multiplication by~$\tau$ in the order, so we
define~$\Phi$ as block diagonal matrix by
\begin{equation*}
  \Phi \colonequals 
  \diag{\sigma_1(\tau), \ldots, \sigma_s(\tau), 
    {\footnotesize
    \begin{pmatrix}
      \Re \sigma_{s+1}(\tau)&-\Im\sigma_{s+1}(\tau)\\
      \Im\sigma_{s+1}(\tau)&\Re \sigma_{s+1}(\tau)
    \end{pmatrix}},
    \ldots,
    {\footnotesize
    \begin{pmatrix}
      \Re \sigma_{s+t}(\tau)&-\Im\sigma_{s+t}(\tau)\\
      \Im\sigma_{s+t}(\tau)&\Re \sigma_{s+t}(\tau)
    \end{pmatrix}}}.
\end{equation*}
The eigenvalues of~$\Phi$ are exactly the conjugates of~$\tau$,
therefore all eigenvalues have absolute value~$\rho$. For the norm
$\norm{\,\cdot\,}$ we choose the Euclidean norm
$\norm{\,\cdot\,}_2$. Then the corresponding operator norm fulfils
\begin{equation*}
  \norm{\Phi} 
  = \max\set*{\abs{\f{\sigma_j}{\tau}}}{j\in\set{1,2,\dots,s+t}}
  = \rho.
\end{equation*}
In the same way we get $\norm{\Phi^{-1}} = \rho^{-1}$. 

Now let $T\subseteq\R^n$ be a set which tiles the~$\R^n$ by the
lattice~$\Lambda$, choose~$w$ as in the set-up in Section~\ref{sec:set-up}, and
let $\cD$ be a reduced residue digit set modulo~$\Phi^w$ corresponding to the
tiling, cf.\ also Heuberger and
Krenn~\cite{Heuberger-Krenn:2013:general-wnaf}. Since our lattice~$\Lambda$
comes from the order~$\Z[\tau]$ and our map~$\Phi$ corresponds to the
multiplication by~$\tau$ map, the size of the digit set~$\cD$ is $\rho^{n(w-1)}
\left(\rho^n-1\right) + 1$, see~\cite{Heuberger-Krenn:2012:wnaf-analysis} for details.

Since our set-up, see Section~\ref{sec:set-up}, is now complete, we get that
Theorem~\ref{thm:countdigits} holds. We want to restate this for our special
case of $\tau$\nbd-adic \wNAF{}-expansions. This is done in
Corollary~\ref{cor:count-tau-adic}. To prove periodicity of the
function~$\psi_\eta$ in that corollary, we need the following lemma.


\begin{lemma}\label{lem:cond-periodic-hyperll}
  Suppose
  \begin{equation*}
    \Phi = Q \diag{\rho e^{i\theta_1}, \dots, \rho e^{i\theta_n}} Q^{-1},
  \end{equation*}
  where $Q$ is a regular matrix and let $U=\ball{0}{1}$ be the unit ball. Then
  \begin{equation*}
    Q \diag{e^{i\theta_1}, \dots, e^{i\theta_n}} Q^{-1} U = U.
  \end{equation*}
\end{lemma}


\begin{proof}
  Since $\Phi$ is normal, the matrix $Q \diag{e^{i\theta_1}, \dots,
    e^{i\theta_n}} Q^{-1}$ is unitary. Therefore balls are mapped to balls
  bijectively, which was to be proved.
\end{proof}


Now, as mentioned above, we reformulate Theorem~\ref{thm:countdigits} for our
$\tau$\nbd-adic set-up. This gives the following corollary.


\begin{corollary}\label{cor:count-tau-adic}
  Let $\tau$ be an algebraic integer, where all conjugates have the same
  absolute value, denote the embeddings of $\Q(\tau)$ by $\sigma_1$, \dots,
  $\sigma_{s+t}$ as above, and define a norm by $\norm{z}^2 = \sum_{i=1}^{s+t}
  d_i \abs{\f{\sigma_i}{z}}^2$ with $d_1=\dots=d_s=1$ and
  $d_{s+1}=\dots=d_{s+t}=2$. 

  Let $0 \neq \eta \in \cD$ and $N\in\R$ with $N>0$. We denote the number of
  occurrences of the digit $\eta$ in all width\nbd-$w$ non-adjacent forms in
  $\Ztau$, where the norm of its value is smaller than $N$, by
  \begin{equation*}
    \f{Z_\eta}{N}
    = \sum_{\substack{z \in \Ztau \\ \norm{z}<N}} \sum_{j\in\N_0} 
    \iverson*{$j$th digit of $z$ in its \wNAF{}-expansion equals $\eta$}.
  \end{equation*}
  Then we get  
  \begin{equation*}
      \f{Z_\eta}{N} 
      = N^n \frac{\pi^{n/2}}{\f{\Gamma}{\frac{n}{2}+1}} E \log_\rho N
      + N^n \f{\psi_\eta}{\log_\rho N}
      + \Oh{N^\beta \log_\rho N},
  \end{equation*}
  where we have the constant of the expectation
  \begin{equation*}
    E = \frac{1}{\rho^{n(w-1)}((\rho^n-1)w+1)},    
  \end{equation*}
  cf.\ Theorem~\vref{th:w-naf-distribution}, a function $\f{\psi_\eta}{x}$
  which is \periodic{1} and continuous and $\beta<n$.
\end{corollary}

\begin{proof}
  We choose $U = \ball{0}{1}$ the unit ball in the $\R^n$. Then $U$ is
  measurable, $d=1$ and $\delta = n-1 < n$. Further the $n$\nbd-dimensional
  Lebesgue measure of~$U$ equals $\frac{\pi^{n/2}}{\f{\Gamma}{\frac{n}{2}+1}}$.
  The condition $\cardV{\boundaryV{NU}} = \Oh{N^\delta}$ can be checked
  easily. In the case of a quadratic~$\tau$ this is done
  in~\cite{Heuberger-Krenn:2012:wnaf-analysis}. The periodicity (and therefore
  continuity) of $\psi_\eta$ follows from
  Lemma~\ref{lem:cond-periodic-hyperll}. We can choose
  $\beta=\max\set{\alpha,n-1}$.
\end{proof}

\let\varhexagon\origvarhexagon


\renewcommand{\MR}[1]{}

\bibliographystyle{amsplain}
\bibliography{cheub}


\end{document}

